\documentclass[reqno,11pt]{amsart}

\usepackage{amsmath, amsthm, amssymb, amsfonts, amsthm}
\usepackage{amsxtra, amscd, geometry, graphicx,epic,eepic}
\usepackage[all,cmtip]{xy}
\usepackage{mathrsfs}
\usepackage{hyperref}
\usepackage{graphicx, nicefrac}

\newtheorem{theorem}{Theorem}         
\newtheorem{proposition}[theorem]{Proposition}
\newtheorem{cor}[theorem]{Corollary}
\newtheorem{lemma}[theorem]{Lemma}

\newtheorem{definition}[theorem]{Definition}

\theoremstyle{definition}
\newtheorem{remark}[theorem]{Remark}

\newtheorem*{rmk*}{Remark}
\newtheorem*{Ex1*}{Example 1}
\newtheorem*{Ex2*}{Example 2}
\newtheorem*{Ex3*}{Example 3}
\newtheorem*{Ex4*}{Example 4}
\newtheorem*{Ex5*}{Example 5}
\newtheorem*{Ex6*}{Example 6}

\numberwithin{equation}{section}

\newcommand{\N}{\mathbb N}
\newcommand{\Q}{\mathbb Q}
\newcommand{\R}{\mathbb R}
\newcommand{\Z}{\mathbb Z}

\renewcommand{\H}{\mathbb H}

\newcommand{\e}{E}
\renewcommand{\b}{B}

\newcommand{\PP}{\mathscr P}
\newcommand{\SSS}{\mathscr S}
\newcommand{\X}{\mathscr X}
\newcommand{\TT}{\mathscr T}
\newcommand{\FFF}{\mathscr F}
\newcommand{\WW}{\mathscr W}
\newcommand{\NN}{\mathscr N}

\newcommand{\RRR}{\mathscr R}
\newcommand{\AAA}{\mathscr A}
\newcommand{\Te}{T_E}
\newcommand{\oTe}{\overline{T}_{E}}
\newcommand{\Tb}{T_B}
\newcommand{\oTb}{\overline{T}_{B}}

\newcommand{\ep}{\operatorname{eper}(\omega)}
\newcommand{\llangle}{\langle\! \langle}
\newcommand{\rrangle}{\rangle\! \rangle}
\newcommand{\llb}{[\! [}
\newcommand{\rrb}{]\! ]}
\newcommand{\wTheta}{\widetilde{\Theta}}

\title[Periodic points of certain Gauss shifts with infinite invariant measure]{Distribution of periodic points of certain Gauss shifts with infinite invariant measure}
\author{Florin P. Boca and Maria Siskaki}

\address{Department of Mathematics, University of Illinois at Urbana-Champaign,
Urbana, IL 61801}

\address{E-mail: fboca@illinois.edu \quad siskaki2@illinois.edu}


\date{\today}

\begin{document}


\begin{abstract}
This paper investigates the periodic points of the Gauss type shifts associated to the
even continued fraction (Schweiger) and to the backward continued fraction (R\' enyi).
We show that they coincide exactly with two sets of quadratic irrationals that we call
$E$-reduced, and respectively $B$-reduced. We prove that these numbers are
equidistributed with respect to the (infinite) Lebesgue absolutely continuous
invariant measures of the corresponding Gauss shift.
\end{abstract}

\maketitle

\section{Introduction}
Euclidean algorithms and their associated continued fraction expansions generate interesting examples of
(non-invertible) measure preserving transformations, called Gauss shifts.
The best known is the Gauss map
\begin{equation*}
T:[0,1) \longrightarrow [0,1),\quad T(x):=\bigg\{ \frac{1}{x}\bigg\} =\frac{1}{x}-\bigg\lfloor \frac{1}{x}\bigg\rfloor\ \
\mbox{\rm if $x\neq 0$}, \qquad T(0):=0,
\end{equation*}
associated with the RCF (regular continued fraction) expansion
\begin{equation*}
x=[a_1,a_2,\ldots]:=\cfrac{1}{a_1+\cfrac{1}{a_2+\cfrac{1}{\ddots}}}\, ,\qquad a_i \in \N .
\end{equation*}
On such expansions $T$ acts as a one-sided shift
\begin{equation}\label{eq1.1}
T([a_1,a_2,\ldots ])=[a_2,a_3,\ldots ].
\end{equation}
The digits of $x$ are fully recaptured by the $T$-iterates of $x$ as $a_1 =\big\lfloor \frac{1}{x}\big\rfloor$,
$a_{n+1}=\big\lfloor \frac{1}{T^n (x)}\big\rfloor$, $n\geq 1$. It was discovered by Gauss that the probability
measure $\mu_G:=\frac{dx}{(1+x)\log 2}$ is $T$-invariant. The endomorphism $T$ is exact in the sense of Rohlin.
The measure $\mu_G$ is the unique Lebesgue absolutely continuous $T$-invariant probability measure. For a comprehensive
presentation of the ergodic properties of $T$ we refer to \cite{KMS}.

Equality \eqref{eq1.1} shows that the periodic points of $T$ are precisely the \emph{reduced} quadratic irrationals
in $[0,1)$, i.e. the numbers with periodic RCF-representation $\omega =[\, \overline{a_1,\ldots,a_n}\,]$.
These are known to coincide with the \emph{QIs} (quadratic irrationals) $\omega \in [0,1)$ with conjugate $\omega^* \in (-\infty,-1]$.
An important connection between $\omega$ and $\omega^*$ is provided by the Galois formula (\cite{Ga})
\begin{equation}\label{eq1.2}
[\,\overline{a_1,\ldots ,a_n}\,]^* =-\frac{1}{[\,\overline{a_n,\ldots,a_1}\,]} \, .
\end{equation}
Reduced QIs are naturally ordered by their \emph{length}
\begin{equation*}
\varrho (\omega):=2\log \epsilon_0 (\omega) ,
\end{equation*}
where $\epsilon_0(\omega)=\epsilon_\Delta = \frac{1}{2}(u_0+v_0\sqrt{\Delta})$ is the fundamental solution of the Pell equation
$u^2 -\Delta v^2 =4$, with $\Delta:=\operatorname{disc}(\omega)$. Geometrically, $\varrho (\omega)$ measures the length of the closed primitive
geodesic on the modular surface ${\mathscr M}=\operatorname{SL}(2,\Z)\backslash {\mathbb H}$, which has a lift to ${\mathbb H}$ with endpoints at
$\omega^{-1}=[\, \overline{a_1,\ldots ,a_n}\, ]^{-1}$ and $\omega^* =-[\, \overline{a_n,\ldots ,a_1}\,]$. More concretely,
if we set
\begin{equation}\label{eq1.3}
\Omega (\omega):=\left( \begin{matrix} a_1 & 1 \\ 1 & 0 \end{matrix}\right) \cdots\left( \begin{matrix} a_n & 1 \\ 1 & 0 \end{matrix} \right) ,
\qquad \widetilde{\Omega}(\omega):=\begin{cases}
\Omega(\omega) & \mbox{\rm if $n$ even} \\ \Omega(\omega)^2 & \mbox{\rm if $n$ odd,}
\end{cases}
\end{equation}
and denote by ${\mathfrak r}(\sigma)$ the spectral radius of a $2\times 2$ matrix $\sigma$, then
\begin{equation}\label{eq1.4}
\varrho (\omega) =2\log {\mathfrak r}(\widetilde{\Omega}(\omega)).
\end{equation}
It is well known (see, e.g., \cite{Fa}) that, in the upper half-plane $\H$, for every
$\sigma \in \operatorname{SL}(2,\R)$ and $z\in\H$ on the axis of $\sigma$,
\begin{equation*}
d(z,\sigma z)=2\log {\mathfrak r}(\sigma).
\end{equation*}

Employing Mayer's thermodynamic formalism for the Gauss shift $T$ (\cite{Ma}) and the Series coding of geodesics on
the modular surface ${\mathscr M}$ (\cite{Se}),
Pollicott proved (\cite{Po}, see also Faivre's ensuing work \cite{Fa}) that the periodic points of $T$ are equidistributed with respect to the Gauss
measure $\mu_G$, and also that closed geodesics on
$\mathscr M$ are uniformly distributed when ordered by length.
More recently, Kelmer proved a more general result (\cite{Ke}) about closed geodesics with prescribed linking
number and the uniform distribution on $[0,1)^2$ of the periodic points of the
(invertible) natural extension $\widetilde{T}$ of $T$, with respect to the $\widetilde{T}$-invariant probability measure $\widetilde{\mu}_G := \frac{dx dy}{(xy+1)^2 \log 2}$.
These proofs rely on the spectral analysis of the
nuclear Perron-Frobenius operator associated to $T$, acting on the disk algebra $A(\{ \lvert z-1\rvert \leq \frac{3}{2} \})$,
and ultimately on an application of the Wiener-Ikehara tauberian theorem, which does not lead to effective estimates for the error term
in the final asymptotic formula. Effective asymptotic results from applications of transfer operators have very recently emerged in the study of the additive cost of moderate growth of
reduced QIs (\cite{CV}), and respectively of the average of word lengths of closed geodesics on negatively curved surfaces (\cite{CP}).

When ordering by discriminant, a powerful number theoretical result of Duke (\cite{Du}) shows that the collection of closed geodesics with
the same discriminant $\Delta >0$, and hence with the same length $2\log \epsilon_\Delta$, are equidistributed in
${\mathscr M}$.

A more direct number theoretical approach for estimating the number of periodic points of $T$, initiated in \cite{KO}, was further sharpened
by one of the authors (\cite{Bo}), followed by work of Ustinov \cite{Us} (see also the Appendix to \cite{He}).
The approach from \cite{Bo} and \cite{Us} relies essentially on applications of the Weil bound for Kloosterman sums.
In that setting, the problem was reduced to deriving an asymptotic formula for the number $S(\alpha,\beta;N)$ of matrices
$\Big( \begin{smallmatrix} p & p^\prime \\ q & q^\prime \end{smallmatrix}\Big) \in \operatorname{SL}(2,\Z)$ subject to the inequalities
$\alpha q^\prime \geq p^\prime >p>0$, $\beta q^\prime \geq q> p$, and $p +q^\prime \leq N$,
where $\alpha,\beta \in [0,1]$ are fixed and $N\rightarrow \infty$. The following effective estimate was proved in \cite{Us}:
\begin{equation*}
\begin{split}
S(\alpha,\beta;N) & =\frac{\log (\alpha\beta+1)}{2\zeta(2)}\, N^2 + O_\varepsilon (N^{3/2+\varepsilon}) \\ & =
\frac{N^2}{2\zeta(2)} \iint_{[0,\alpha]\times [0,\beta]} \frac{dx\, dy}{(xy+1)^2} +O_\varepsilon (N^{3/2+\varepsilon}) ,
\qquad \forall \varepsilon >0.
\end{split}
\end{equation*}
Elementary considerations (\cite{KO,Bo,Us}) then lead to the estimate
\begin{equation}
\sum\limits_{\substack{\omega \, \text{\rm reduced Q.I.}\\ \varrho (\omega) \leq R \\ 0\leq \omega \leq \alpha \\ 0\leq -1/\omega^* \leq \beta}} 1 =
\frac {e^R}{2\zeta(2)} \iint_{[0,\alpha]\times [0,\beta]} \frac{dx\, dy}{(xy+1)^2}
+O_\varepsilon (e^{(3/4+\varepsilon)R}),
\end{equation}
showing that the periodic points of the natural extension of $T$ are $\widetilde{\mu}_G$-equidistributed.

In Vall\' ee's  classification of Euclidean algorithms (\cite{V1,V2}), the MSB (most significant bits) class is given special attention.
There are six CF (continued fraction) MSB algorithms, denoted by (G), (M), (K), (E), (O), (T). In our terminology they are:
(G) $\longleftrightarrow$ RCF (regular CF), (M) $\longleftrightarrow$ BCF (backward CF), (K) $\longleftrightarrow$ NICF (nearest integer CF),
(E) $\longleftrightarrow$ ECF (even CF), (O) $\longleftrightarrow$ OCF (odd CF), (T) $\longleftrightarrow$ LCF (Lehner CF).
The Gauss map corresponding to type (T) is given by $V(x):=\frac{x}{1-x}$ if $x\in [0,\frac{1}{2}]$ and $V(x):=
\frac{1-x}{x}$ if $x\in [\frac{1}{2},1]$. This is the familiar Farey map on $[0,1]$. It was observed in \cite{DK} that conjugating by
$x\mapsto x+1$ one gets the Gauss map of the Lehner CF on the interval $[1,2]$, which involves only the digits $(1,+1)$ and $(2,-1)$ (see also \cite{Me} for
a geometric approach).
The algorithms (M), (E) and (T) are ``slow" and belong to the ``Bad Class" (see Section 2.5 of \cite{V2}). Incidentally,
their associated Gauss shifts have infinite invariant measure, which makes a Perron-Frobenius operator approach as in \cite{BaV,He,Ke,Po} more challenging.

The analogue of Pollicott's problem for the Farey map has been already thoroughly studied by Heersink (\cite{He}, see also \cite{PU} for a broader scenery).
Building on the approach from \cite{Po} and \cite{Fa}, the equidistribution of the periodic
points of the Farey map, and also of its natural extension, with respect to their (infinite) invariant measures have been established in \cite{He}.

In this paper we investigate the distribution of the periodic points of the Gauss shifts $T_E$ in situation (E) and $T_B$ in the situation (M).
Our results show that, when ordered by appropriate lengths $\varrho_E$ and respectively $\varrho_B$, these
subsets of QIs are equidistributed in an effective manner with respect to the Lebesgue absolutely continuous measure of
the corresponding Gauss shift. In fact, Theorems 1 and 4 below show that the periodic points of the natural extensions of
these maps are equidistributed with respect to their invariant measure.
With the purpose of stating these results, we start with a summary of definitions and properties
of the shifts $T_E$ and $T_B$, and of their periodic points.
Here, we prefer to work with ECF and BCF-expansions of
numbers in $[1,\infty)\setminus \Q$ rather than $[0,1]\setminus \Q$. Definitions and results can be easily formulated on
$[0,1]$ by conjugating by $x\mapsto \frac{1}{x}$. An operator theoretical approach appears to be complicated, a first difficulty being to find
an appropriate invariant space of analytic functions under the corresponding Perron-Frobenius operator. Our approach is number theoretical and
ultimately relies on the Weil bound for Kloosterman sums.

Every irrational number $u>1$ has a unique ECF-expansion
\begin{equation}\label{eq1.6}
u=[ (a_1,e_1),(a_2,e_2),\ldots ]:=
a_1+\cfrac{e_1}{a_2 +\cfrac{e_2}{a_3+\cdots}} \geq 1,
\end{equation}
where $a_i\in 2\N$ and $e_i\in \{ \pm 1\}$.
The corresponding ECF Gauss shift $\Te$ acts on $[1,\infty)\setminus \Q$ by $\Te ([(a_1,e_1),(a_2,e_2),\ldots ]) =[(a_2,e_2),(a_3,e_3), \ldots ]$. In
different notation we have
\begin{equation*}
\Te (u)= \left( \begin{matrix} 0 & e_1 \\ 1 & -a_1 \end{matrix} \right)u =
\frac{e_1}{u-a_1} \, ,
\end{equation*}
where $a_1=a_1(u) =2\big\lfloor \frac{u+1}{2}\big\rfloor \in 2\N$ and $e_1=e_1(u) =\operatorname{sgn} ( u-a_1(u))\in \{ \pm 1\}$.
The infinite measure $\mu_{E} =\big( \int_{-1}^1 (u+v)^{-2} dv\big) du=\frac{2du}{(u-1)(u+1)}$ is $\Te$-invariant.
Conjugating by $J(x):=\frac{1}{x}$, one gets the customary ECF Gauss map
$\oTe:=J^{-1} \Te J$ (\cite{KL,S1,S2}), which acts on $[0,1]$ as
\begin{equation}\label{eq1.7}
\oTe (x)=\bigg| \frac{1}{x} -2\bigg[ \frac{x+1}{2x} \bigg]\bigg| \  \  \mbox{\rm if $x\neq 0$},\qquad \oTe (0):=0,
\end{equation}
with invariant measure $\nu_{E} =J_\ast \mu_{E} = \frac{2dx}{(1-x)(1+x)}$.
Equivalently, $\overline{T}_E$ acts as a shift on ECF-expansions
\begin{equation*}
\oTe ([ (a_1,e_1),(a_2,e_2),(a_3,e_3),\ldots ]) =[(a_2,e_2),(a_3,e_3),\ldots ] .
\end{equation*}

The periodic points of $T_E$ are exactly the irrationals with periodic
ECF-expansion $\omega =[\, \overline{(a_1,e_1),\ldots,(a_n,e_n)}\,]$.
In Section \ref{sectECF} we will show that
these are also exactly the elements of the set $\RRR_E$ of QIs
$\omega >1$ with $\omega^* \in [-1,1]$, which we call \emph{$E$-reduced} QIs. 

To define the length of
$\omega=[\, \overline{(a_1,e_1),\ldots,(a_n,e_n)}\,] \in \RRR_E$ with $n=\operatorname{per}(\omega)$, we introduce the matrices
\begin{equation}\label{eq1.8}
\Omega_E(\omega):= \left( \begin{matrix} a_1 & e_1 \\ 1 & 0 \end{matrix}\right) \cdots
\left( \begin{matrix} a_n & e_n \\ 1 & 0 \end{matrix}\right),\quad
\widetilde{\Omega}_E(\omega):= \begin{cases}
\Omega_E(\omega) & \mbox{\rm if $(-e_1)\cdots (-e_n)=+1$} \\
\Omega_E(\omega)^2 & \mbox{\rm if $(-e_1)\cdots (-e_n)=-1$,} \end{cases}
\end{equation}
then set
\begin{equation*}
\varrho_E(\omega):=2\log {\mathfrak r} ( \widetilde{\Omega}_E(\omega)),\qquad \omega\in\RRR_E.
\end{equation*}

Denote
\begin{equation*}
I_2 :=\left( \begin{matrix} 1 & 0 \\ 0 & 1 \end{matrix}\right) ,\qquad
J_2 :=\left( \begin{matrix} 0 & 1 \\ 1 & 0 \end{matrix}\right) .
\end{equation*}
Consider the Theta groups $\widetilde{\Theta}:=\{ \sigma \in \operatorname{GL}(2,\Z) \mid \sigma \equiv I_2 \ \mbox{\rm or}\
J_2 \pmod{2} \}$ and $\Theta:= \widetilde{\Theta} \cap \operatorname{SL}(2,\Z)$.
In the RCF case, the stabilizer $\{ \sigma\in \operatorname{GL}(2,\Z) \mid \sigma \omega=\omega\}$ of a reduced quadratic irrational
$\omega$ is used to produce solutions of the Pell equations $u^2-\Delta v^2=\pm 4$ (see. e.g., \cite{Ha,On}). The connection obtained
by replacing reduced QIs by $E$-reduced QIs and the group $\operatorname{GL}(2,\Z)$ by
its subgroup $\widetilde{\Theta}$ will be discussed in Section \ref{Pell}.

The closed primitive geodesics on the modular surface $\Theta \backslash \H$ correspond exactly (cf. \cite{BM}) to
those closed geodesics that lift to a geodesic on $\H$ with endpoints
\begin{equation*}
\gamma_{+\infty} =e [ \,\overline{(a_1,e_1),\ldots ,(a_n,e_n)}\,]=e\omega \quad \mbox{\rm and}
\quad \gamma_{-\infty}  = \gamma_{+\infty}^\ast,
\end{equation*}
 for some $e \in \{ \pm 1\}$, and also with
\begin{equation*}
(-e_1) \ldots (-e_n)=1 .
\end{equation*}
This shows that the $E$-reduced QIs $\omega$ are naturally ordered by $\varrho_E (\omega)$,
which represents the length of such a geodesic.

The difference between $E$-reduced QIs and ordinary reduced QIs will be illustrated in the Appendix,
where we consider the families of $E$-reduced QIs of the form $[\,\overline{(a,-1)}\,]$,
$[\,\overline{(a_1,1),(a_2,-1)}\,]$,
$[\,\overline{(a_1,-1),(a_2,1)}\,]$, and
$[\,\overline{(a_1,-1),(a_2,-1)}\,]$. We compute their discriminants and lengths $\varrho_E(\omega)$,
and show that only the third family contains regular reduced QIs.

For every $\alpha,\beta_1,\beta_2 \geq 1$, $N\in \N$, set
\begin{equation}\label{eq1.9}
r_{\e} (\alpha,\beta_1,\beta_2;R) :=\sum\limits_{\substack{\omega \in \RRR_{\e}, \, \varrho_E (\omega) \leq R \\
\omega \geq \alpha,\, -\frac{1}{\beta_2} \leq \omega^* \leq \frac{1}{\beta_1} }} 1.
\end{equation}
We will prove the following results concerning the distribution of periodic points of $T_E$:

\begin{theorem}\label{thm1}
For every $\alpha,\beta_1,\beta_2 \geq 1$ with $(\alpha, \beta_1) \neq (1,1)$,
\begin{equation*}
r_E (\alpha,\beta_1,\beta_2;R) \
= \  C(\alpha,\beta_1,\beta_2) e^R +O_{\alpha,\beta_1,\varepsilon} (e^{(3/4+\varepsilon)R}),\qquad \forall \varepsilon >0,
\end{equation*}
where
\begin{equation*}
C(\alpha,\beta_1,\beta_2) = \frac{1}{\pi^2}
\log \bigg( \frac{\alpha\beta_2+1}{\alpha\beta_2}\cdot \frac{\alpha\beta_1}{\alpha\beta_1-1}\bigg)
=\frac{1}{\pi^2} \iint_{[\alpha,\infty) \times [-\frac{1}{\beta_1},\frac{1}{\beta_2}]} \frac{du\, dv}{(u+v)^2} .
\end{equation*}
\end{theorem}

Taking $\beta_1=\beta_2=1$, and $\alpha=\beta_2=1$, $\beta_1=\infty$ respectively, we find

\begin{cor}\label{cor2}
For every $\alpha >1$,
\begin{equation*}
\sum\limits_{\substack{\omega\in \RRR_{\e},\, \omega \geq \alpha \\ \varrho_E (\omega)\leq R}} 1
=\frac{e^R}{\pi^2} \int_\alpha^\infty d\mu_E +O_{\alpha,\varepsilon} (e^{(3/4+\varepsilon)R}) .
\end{equation*}
\end{cor}

\begin{cor}\label{cor3}
$\displaystyle \quad
\sum\limits_{\substack{\omega\in \RRR_{\e},\, \omega^* <0 \\ \varrho_E(\omega)\leq R}} 1
=\frac{e^R \log 2}{\pi^2} +O_\varepsilon (e^{(3/4+\varepsilon)R}) .$
\end{cor}

Every number $u\in [1,\infty)\setminus \Q$ has a unique BCF-expansion
\begin{equation}\label{eq1.10}
u = \llb a_1,a_2,a_3,\ldots \rrb:=  a_1-\cfrac{1}{a_2 -\cfrac{1}{a_3-\cfrac{1}{\ddots}}} \geq 1 ,\qquad a_i \in \N,\ a_i \geq 2.
\end{equation}
The BCF Gauss shift $\Tb$ acts on $[1,\infty)\setminus \Q$ as $\Tb (\llb a_1,a_2,\ldots\rrb) =\llb a_2,a_3,\ldots\rrb$, or,
in different notation,
\begin{equation*}
\Tb (u)= M(a_1,-1)^{-1} u =\left( \begin{matrix} 0 & -1 \\ 1 & -a_1 \end{matrix} \right)u =
\frac{1}{a_1-u}  =\frac{1}{1-\{ u\}},
\end{equation*}
where $a_1=a_1(u) =1+\lfloor u\rfloor\geq 2$.
The infinite measure $\mu_{B} =\big( \int_0^1 (u-v)^{-2} dv\big) du =\frac{du}{u(u-1)}$ is $\Tb$-invariant.
Conjugating by $J_{B} (x):=\frac{1}{1-x}$, one gets the R\' enyi-Gauss map
$\oTb:=J_{B}^{-1} \Tb J_{\b}$, which acts on $[0,1]$ as
\begin{equation}\label{eq1.11}
\oTb (x)=\bigg\{ \frac{1}{1-x}\bigg\} ,\qquad x\neq 1,
\end{equation}
with invariant measure $\nu_{\b} =J_{\b\ast} \mu_{B} = \frac{dx}{x}$ (\cite{Re,AF}).
BCF-expansions of rational numbers arise in the study of singularities of complex manifolds (\cite{Hi})
and in formulae for class numbers (\cite{Za}).

When performing elementary computations with backward continued fractions, one can simply take
$e_i=-1$, $\forall i\geq 1$ in the ECF-expansions and assume that $a_i \geq 2$ are (not necessarily even) integers.
In particular, this shows that the sequence $(\frac{p_k}{q_k})$ of convergents of a given number is decreasing.
The periodic points of $T_B$ are exactly the irrationals with periodic
BCF-expansion $\omega =\llb\, \overline{a_1,\ldots,a_n}\,\rrb$. These are shown to also coincide with the elements of the set $\RRR_B$ of QIs
$\omega >1$ with $\omega^* \in [0,1]$, which we call \emph{$B$-reduced} QIs. Consider the matrix
\begin{equation*}
\Omega_B (\omega):= \left( \begin{matrix} a_1 & -1 \\ 1 & 0 \end{matrix}\right) \cdots
\left( \begin{matrix} a_n & -1 \\ 1 & 0 \end{matrix}\right),\quad \mbox{\rm where
$n=\operatorname{per}(\omega)$.}
\end{equation*}
Notice that $\det (\Omega_B(\omega))=+1$, so $\widetilde{\Omega}_B (\omega)=\Omega_B (\omega)$ for all $\omega\in\RRR_B$. Define
\begin{equation*}
\varrho_B(\omega):=2\log {\mathfrak r} ( \Omega_B (\omega)),\qquad \omega\in\RRR_B.
\end{equation*}
Define also
\begin{equation*}
r_B(\alpha,\beta;R):= \sum\limits_{\substack{\omega\in \RRR_{\b},\,  \varrho_B (\omega) \leq R \\
\omega \geq \alpha,\, 0< \omega^* \leq \frac{1}{\beta}}} 1.
\end{equation*}
We will prove the following results concerning the distribution of periodic points of $T_B$:

\begin{theorem}\label{thm4}
For every $\alpha,\beta \geq 1$ with $(\alpha,\beta)\neq (1,1)$,
\begin{equation*}
r_B(\alpha,\beta;R) =
\frac{e^R}{2\zeta(2)} \iint_{[\alpha,\infty) \times [0,\frac{1}{\beta}]} \frac{du\, dv}{(u-v)^2}
+O_{\alpha,\beta_1,\varepsilon} (e^{(3/4+\varepsilon)R}),\qquad \forall \varepsilon >0.
\end{equation*}
\end{theorem}

Taking $\beta=1$ we find

\begin{cor}\label{cor5}
For every $\alpha >1$,
\begin{equation*}
\begin{split}
\sum\limits_{\substack{\omega\in \RRR_{\b},\, \omega \geq \alpha \\ \varrho_B (\omega) \leq R}} 1 & =
\frac{e^R}{2\zeta(2)} \log \bigg(\frac{\alpha}{\alpha-1}\bigg) +O_\varepsilon (e^{(3/4+\varepsilon)R}) \\  & =
\frac{e^R}{2\zeta(2)} \int_{[\alpha,\infty)} \frac{du}{u(u-1)} +O_{\alpha,\varepsilon} (e^{(3/4+\varepsilon)R}).
\end{split}
\end{equation*}
\end{cor}

\begin{remark}\label{rem6}
A shorter proof can be achieved without the analytic number theoretical estimates \eqref{eq5.7} and \eqref{eq5.8}, but
with the price of getting an error term $O_\varepsilon (e^{(7/8+\varepsilon)R})$ instead of
$O_\varepsilon (e^{(3/4+\varepsilon)R})$ in Theorems \ref{thm1} and \ref{thm4}.
\end{remark}

We expect our approach to also work in the remaining situations of (good) MSB Euclidean algorithms of type (K) and (O).
We are planning to study this in further work.

\section{Even Continued Fractions}\label{sectECF}
In this section we investigate various algebraic properties of ECF expansions. In Subsection 2.1 we revisit the main features of
the ECF Gauss shift, and in Subsection 2.2 we focus on QIs. Among other things, we give a short algebraic proof of the ECF analogue of the Galois formula \eqref{eq1.2}
(previously proved in different ways in \cite{KL} and \cite{BM}), then prove
that ECF-periodic QIs coincide exactly with QIs $\omega >1$ having $\omega^* \in (-1,1)$
(previously proved geometrically in \cite{BM}). Finally, inspired by the approach pursued in the RCF situation in \cite{KO,Bo,Us}, we describe
some bijections between subsets of ${\mathscr R}_E$ and subsets of $\operatorname{SL}(2,\Z)$ that will play an essential role in the proof of Theorem \ref{thm1}.

\subsection{ECF expansions and the associated Gauss shift}
The ECF-convergents $\frac{p_k}{q_k}$ of an irrational number $u$ as in \eqref{eq1.6} are given by
\begin{equation*}
\begin{split}
& p_0=p_0(u):=1,\quad p_1=p_1(u):=a_1 ,\quad p_k=p_k(u):=a_k p_{k-1}+e_{k-1}p_{k-2} ,\\
& q_0=q_0(u):=0,\quad q_1=q_1(u):=1 ,\quad q_k=q_k(u):=a_k q_{k-1}+e_{k-1}q_{k-2} .
\end{split}
\end{equation*}
These relations show
\begin{equation*}
q_k \ \mbox{\rm even}\ \Longleftrightarrow \ k \ \mbox{\rm even} \ \Longleftrightarrow \ p_{k-1}\ \mbox{\rm even} .
\end{equation*}
The matrix $M(a,e):=\Big(\begin{smallmatrix} a & e \\ 1 & 0 \end{smallmatrix}\Big)$ acts on $u$ by
\begin{equation*}
M(a,e) u=a+\frac{e}{u}=[(a,e),u] .
\end{equation*}

The Gauss map $\Te$ acts on $u=[ (a_1,e_1),(a_2,e_2),\ldots ]$ by
\begin{equation*}
\Te (u)= M(a_1,e_1)^{-1} u =\left( \begin{matrix} 0 & e_1 \\ 1 & -a_1 \end{matrix} \right)u =
\frac{e_1}{u-a_1}  .
\end{equation*}
The natural extension of the endomorphism $\Te$ is the automorphism of $[1,\infty)\times [-1,1]$ given by
\begin{equation*}
\widetilde{T}_{E} (u,v)=\bigg( \Te (u) ,\frac{e_1}{v+a_1}\bigg) =\bigg( \frac{e_1}{u-a_1},\frac{e_1}{v+a_1}\bigg) .
\end{equation*}
The map $\widetilde{T}_{E}$ acts on the digits in the following way:
\begin{equation*}
\widetilde{T}_{E} ([w_1,w_2,\ldots],\llangle w_0,w_{-1},\ldots \rrangle) =
([w_2,w_3,\ldots ],\llangle w_1,w_0,\ldots \rrangle) ,\qquad w_i=(a_i,e_i) \in 2\N \times \{ \pm 1\} ,
\end{equation*}
where
\begin{equation*}
\llangle (b_1,f_1),(b_2,f_2),\ldots \rrangle := \cfrac{f_1}{b_1+\cfrac{f_2}{b_2+\ddots}},\qquad
b_j\in 2\N,\    f_j \in \{ \pm 1\}
\end{equation*}
denotes the dual ECF-expansion of irrationals in $[-1,1]$.

\begin{center}
\begin{figure}
\includegraphics[scale=0.45,bb = 10 0 470 350]{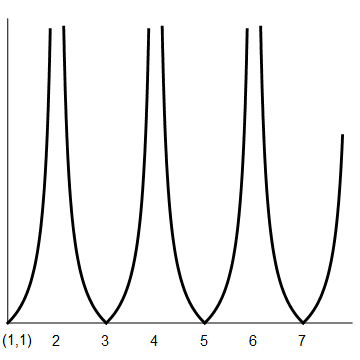}
\includegraphics[scale=0.5,bb = 0 20 280 350]{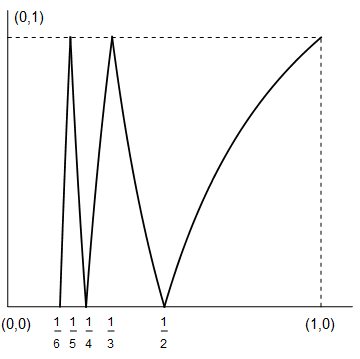}
\caption{The graphs of the maps $\Te$ and $\oTe$ }\label{Figure1}
\end{figure}
\end{center}

From general ergodic theoretical considerations (\cite{BM}), the infinite measure $(u+v)^{-2} du dv$ is
$\widetilde{T}_{E}$-invariant,
while $d\mu_{E} =\big( \int_{-1}^1 (u+v)^{-2} dv\big) du=\frac{2du}{(u-1)(u+1)}$ is $\Te$-invariant.

Conjugating through $J(x):=\frac{1}{x}$, one gets the customary ECF Gauss map
$\oTe:=J^{-1} \Te J$ (\cite{S1,S2,KL,Ce}), which acts on $[0,1]$ as
in formula \eqref{eq1.7},
with invariant measure $\nu_{E} =J_\ast \mu_{E} =
\frac{2dx}{(1-x)(1+x)}$.
The endomorphism $\oTe$ is exact in Rohlin's sense (this can be proved exactly as in the ECF situation using the approach from \cite{ScI}),
and $\nu_{E}$ is the unique $\sigma$-finite Lebesgue absolutely continuous
$\oTe$-invariant measure.

We introduce the matrices
\begin{equation*}
\Omega_k(u):=M(a_1,e_1) \cdots M (a_k,e_k) =\left( \begin{matrix} p_k & p_{k-1} e_k \\
q_k & q_{k-1} e_k \end{matrix}\right) ,
\end{equation*}
with determinant
\begin{equation}\label{eq2.1}
\det (\Omega_k(u)) = (p_k q_{k-1}-p_{k-1} q_k) e_k =(-e_1) \cdots (-e_k) ,
\end{equation}
and inverse
\begin{equation}\label{eq2.2}
\Omega_k (u)^{-1} = (-e_1)\cdots (-e_k) \left( \begin{matrix} q_{k-1} e_k & -p_{k-1} e_k \\
-q_k & p_k \end{matrix} \right) .
\end{equation}
The iterates of $\Te$ can be conveniently described as
\begin{equation}\label{eq2.3}
\Te^k (u) =[(a_{k+1},e_{k+1}) ,(a_{k+2},e_{k+2}),\ldots ] =\Omega_k(u)^{-1} u =
(-e_k) \frac{p_{k-1} -q_{k-1} u}{p_k -q_k u} ,
\end{equation}
leading to
\begin{equation}\label{eq2.4}
\Te (u) \Te^2 (u) \cdots \Te^k (u) =\frac{(-e_1) \cdots (-e_k)}{p_k-q_k u},\qquad \forall k\geq 1 ,
\end{equation}
and to
\begin{equation}\label{eq2.5}
u=\frac{e_k p_{k-1}+\Te^k (u)p_k}{e_k q_{k-1}+\Te^k (u) q_k} .
\end{equation}
Formula \eqref{eq2.3} also gives
\begin{equation*}
q_{k-1} e_k +\Te^k (u)q_k =e_k \bigg( q_{k-1}+\frac{p_{k-1}q_k -uq_{k-1} q_k}{uq_k -p_k} \bigg)
=\frac{(-e_1)\cdots (-e_k)}{p_k -uq_k} ,
\end{equation*}
leading to
\begin{equation*}
u-\frac{p_k}{q_k}= -\frac{(-e_1)\cdots (-e_k)}{q_{k-1}e_k +\Te^k (u) q_k} .
\end{equation*}
As $\Te^k (u)q_k+q_{k-1} e_k > q_k -q_{k-1}\geq 1$, this shows
\begin{equation*}
\operatorname{sgn} \bigg( u-\frac{p_k}{q_k}\bigg) = -(-e_1)\cdots (-e_k),\qquad \forall k\geq 1.
\end{equation*}

\subsection{$E$-reduced quadratic irrationals}
In the sequel we extend some properties of QIs from the RCF
setting (as considered in \cite{KO,On}) to the ECF setting.

\begin{remark}\label{rem7}
The following facts concerning the discriminant and the conjugate of some QIs
$u$ hold for every $\sigma\in \operatorname{GL}(2,\Z)$:
\begin{itemize}
\item[(A)] $\operatorname{disc}(\sigma u)=\operatorname{disc}(u)$.
\item[(B)] $(\sigma u)^* =\sigma u^*$.
\end{itemize}
\end{remark}

\begin{lemma}\label{lem8}
Let $\omega$ be a QI, and let $\sigma=\Big( \begin{smallmatrix} a & b \\ c & d \end{smallmatrix} \Big)
\in \operatorname{GL}(2,\Z)$ such that $\sigma \omega= \omega$.
\begin{itemize}
\item[(i)]
The eigenvalues of $\sigma$ are $\lambda_1 =c \omega +d$,
$\lambda_2 =\frac{\delta}{c\omega+d}$, where $\delta :=ad-bc \in \{ \pm 1\}$.
\item[(ii)]
The eigenvalues of $\sigma^{-1}=\delta \Big( \begin{smallmatrix} d & -b \\ -c & a \end{smallmatrix}\Big)$ are
$\mu_1 =\delta (-c\omega +a)=\frac{1}{\lambda_1}$ and $\mu_2 =\frac{1}{-c\omega +a} =\frac{1}{\lambda_2}$.
\item[(iii)]
$(c\omega+d)^*=c\omega^*+d$.
\end{itemize}
\end{lemma}

\begin{proof}
(i) Clearly $\lambda_1 \lambda_2 =\det (\sigma)$, thus one only has to check that
$c\omega+d +\frac{\delta}{c\omega+d}=a+d$. This is immediately seen to be equivalent to
$c\omega(c\omega+d)=c(a\omega+b)$, or equivalently to $\sigma \omega=\omega$.

(ii) From the first part, $\operatorname{Eig} (\sigma ^{-1}) =\{ \frac{1}{\lambda_1},\frac{1}{\lambda_2}\}$.
More precisely, by direct calculation $\mu_1=\frac{1}{\lambda_1}$ is equivalent to $\sigma \omega=\omega$.

(iii) The number $\eta:=c\omega+d$ satisfies $\eta^2 -(a+d)\eta +ad-bc=0$, so
$\eta+\eta^*=a+d$. On the other hand the equality $c\omega^* +(d-a)\omega^* -b=0$ provides
$\omega +\omega^* =\frac{a-d}{c}$. This leads to $a-c\omega =d+c\omega^*$ and so
$\eta^* =a+d-\eta =c\omega^* +d$.
\end{proof}

Denote by ${\mathfrak N}(u)$ and respectively
${\mathfrak t} {\mathfrak r} (u)$ the norm and trace of the QI $u$ in the associated quadratic field.

\begin{lemma}\label{lem9}
Let $G$ be a subgroup of $\operatorname{GL}(2,\Z)$.
For every QI $\omega$, the map
\begin{equation*}
\Lambda_\omega :G_\omega \longrightarrow \R,\qquad
\Lambda_\omega \left( \begin{matrix} a & b \\ c & d \end{matrix} \right) :=c\omega +d ,
\end{equation*}
defines an injective group homomorphism on $G_\omega :=\{ \sigma\in G \mid \sigma\omega =\omega\}$ with
${\mathfrak N} \circ \Lambda_\omega =\det$ and ${\mathfrak t} {\mathfrak r} \circ \Lambda_\omega =\operatorname{Tr}$.
\end{lemma}

\begin{proof}
Let $\sigma=\Big( \begin{smallmatrix} a & b \\ c & d \end{smallmatrix} \Big) \in G$ and
$\sigma^\prime=\Big( \begin{smallmatrix} a^\prime & b^\prime  \\ c^\prime & d^\prime \end{smallmatrix} \Big) \in G_\omega$.
By definition, $\Lambda_\omega$ maps the product matrix
$\sigma\sigma^\prime$ into $(a^\prime c+c^\prime d)\omega +b^\prime c+d^\prime d$. This coincides with
$(c\omega +d)(c^\prime \omega+d^\prime)=\Lambda_\omega (\sigma)\Lambda_\omega (\sigma^\prime)$ as a result of
$\frac{a^\prime \omega+b^\prime}{c^\prime \omega+d^\prime}=\omega$.

Let $\sigma\in \operatorname{Ker} (\Lambda_\omega)$, so $c\omega+d=1$. The irrationality of $\omega$ yields $c=0$ and $d=1$.
Then $\omega =a\omega +b$, hence $a=1$ and $b=0$.

The equalities ${\mathfrak N} \circ \Lambda_\omega=\det$ and ${\mathfrak t} {\mathfrak r} \circ \Lambda_\omega =\operatorname{Tr}$ follow
from $(c\omega+d)^*=c\omega^*+d$.
\end{proof}

In the sequel we will consider an ECF-periodic QI $\omega =[\,\overline{(a_1,e_1),\ldots,(a_n,e_n)}\,] >1$
with $n=\operatorname{per} (\omega)$. Set
\begin{equation*}
\begin{split}
\delta_n =\delta_n (\omega) & := (-e_1)\cdots (-e_n) , \qquad
\ell=\operatorname{eper}(\omega) := \begin{cases}
n & \mbox{\rm if $\delta_n =+1$} \\
2n & \mbox{\rm if $\delta_n =-1$,} \end{cases} \\
\Omega_E (\omega) & :=\Omega_n(\omega)= M(a_1,e_1) \cdots M(a_n,e_n) =\left( \begin{matrix}
p_n & p_{n-1} e_n \\ q_n & q_{n-1} e_n \end{matrix} \right) , \\
\widetilde{\Omega}_E (\omega) & :=
\begin{cases}
\Omega_E (\omega)  & \mbox{\rm if $\delta_n =+1$} \\
\Omega_E (\omega)^2 & \mbox{\rm if $\delta_n =-1$.} \end{cases}
\end{split}
\end{equation*}
Employing \eqref{eq2.4}, Lemma \ref{lem8}, $\Te^n (\omega) =\omega$, and $\delta_\ell=+1$, we infer that the spectral radius of
$\widetilde{\Omega}_E (\omega)$ is
\begin{equation*}
{\mathfrak r} ( \widetilde{\Omega}_E (\omega))  =\omega \Te(\omega) \cdots \Te^{\ep-1} (\omega) =
q_{\ell} \omega +q_{\ell -1} e_{\ell} =\frac{1}{p_{\ell} -q_{\ell} \omega} >1 .
\end{equation*}
As it will be seen later, formula \eqref{eq2.4} provides an ECF-analogue of Smith's formula (\cite{Sm}).

The following ECF-analogue of the Galois formula \eqref{eq1.2} is known (\cite{BM,KL}).
Here, we give yet another proof, using our setting and the
dual ECF-expansion on $[-1,1]$ (\cite{S1,S2}).

\begin{lemma}\label{lem10}
The conjugate $\omega^\ast$ of $\omega =[ \,\overline{(a_1,e_1),\ldots, (a_n,e_n)}\,] >1$ is given by
\begin{equation*}
-\omega^\ast = \llangle \,\overline{(a_n,e_n),\ldots ,(a_1,e_1)}\,\rrangle \in (-1,1) .
\end{equation*}
\end{lemma}

\begin{proof}
The equality $\Omega_E (\omega) \omega =\omega$ and formulas \eqref{eq2.2} and
\eqref{eq2.3} for $\Omega_E (\omega)^{-1}=\Omega_n(\omega)^{-1}$ show
\begin{equation}\label{eq2.6}
\left( \begin{matrix} q_{n-1} e_n & -p_{n-1} e_n \\ -q_n & p_n \end{matrix} \right) \omega =
\left( \begin{matrix} -q_{n-1} e_n & p_{n-1} e_n \\ q_n & -p_n \end{matrix} \right) \omega =\omega .
\end{equation}
The conjugate $\omega^*$ must also be fixed by the matrix
$\Big( \begin{smallmatrix} -q_{n-1} e_n & p_{n-1} e_n \\ q_n & -p_n \end{smallmatrix} \Big)$, and so
\begin{equation}\label{eq2.7}
\left( \begin{matrix} q_{n-1} e_n & p_{n-1} e_n \\ q_n & p_n \end{matrix} \right) (-\omega^*)=-\omega^*.
\end{equation}

Denote $\widetilde{\omega} :=\llangle \,\overline{(a_n,e_n),\ldots ,(a_1,e_1)}\,\rrangle$. Since
$\Big( \begin{smallmatrix} 0 & e_1 \\ 1 & a_1 \end{smallmatrix} \Big) x=\frac{e_1}{a_1+x}$, we have
$\widetilde{\omega} =\Big( \begin{smallmatrix} 0 & e_n \\ 1 & a_n \end{smallmatrix} \Big) \cdots
\Big( \begin{smallmatrix} 0 & e_1 \\ 1 & a_1 \end{smallmatrix} \Big) \widetilde{\omega} $. Since
$\Big( \begin{smallmatrix} 0 & e_1 \\ 1 & a_1 \end{smallmatrix} \Big) =
\Big( \begin{smallmatrix} q_0 e_1 & p_0 e_1 \\ q_1 & p_1 \end{smallmatrix} \Big) $ and
$\Big( \begin{smallmatrix} 0 & e_k \\ 1 & a_k \end{smallmatrix} \Big)
\Big( \begin{smallmatrix} q_{k-2} e_{k-1} & p_{k-2} e_{k-1} \\ q_{k-1} & p_{k-1} \end{smallmatrix} \Big)=
\Big( \begin{smallmatrix} q_{k-1} e_k & p_{k-1} e_k \\ q_k & p_k \end{smallmatrix} \Big)$, we gather
\begin{equation}\label{eq2.8}
\left( \begin{matrix} q_{n-1} e_n & p_{n-1} e_n \\ q_n & p_n \end{matrix} \right) \widetilde{\omega}=\widetilde{\omega} .
\end{equation}
Suppose that $\widetilde{\omega}\neq -\omega^*$. Then, from \eqref{eq2.7}, \eqref{eq2.8} and \eqref{eq2.6} we infer
$-\omega^* \widetilde{\omega} =-\frac{p_{n-1}e_n}{q_n}=\omega \omega^*$. This leads to a contradiction because $\omega >1 > |\widetilde{\omega}|$,
and so $-\omega^*=\widetilde{\omega}$.
\end{proof}

\begin{definition}\label{def11}
Given $\Delta \in \N$, $\Delta$ not a square, denote
\begin{equation*}
\X(\Delta):=\{ \omega \ \mbox{QI} \mid \operatorname{disc}(\omega)=\Delta\} .
\end{equation*}
The QI $\omega >1$ is called {\em $E$-reduced} if $-1<\omega^* <1$. We denote
\begin{equation*}
\RRR_{\e} (\Delta):=\{ \omega \in \X(\Delta) \mid \omega \ \mbox{$E$-reduced}\}\quad \mbox{and} \quad
\RRR_{\e}:=\bigcup\limits_{\Delta >0} \RRR_{\e} (\Delta) .
\end{equation*}
\end{definition}

\begin{lemma}\label{lem12}
The set $\RRR_{\e} (\Delta)$ is finite.
\end{lemma}

\begin{proof}
Finiteness of the set $\{ \omega \in \X(\Delta) \mid \omega >1, 0> \omega^*>-1\}$ is well known,
so it suffices to show that there are only finitely many numbers $\omega \in \X(\Delta)$ such that
$\omega >1>\omega^* >0$. Let $AX^2+BX+C$ be the minimal polynomial of such $\omega$, with $A>0$ and $(A,B,C)=1$, $\Delta=B^2-4AC$.
From $\omega\omega^*=\frac{C}{A}$ and $\omega +\omega^*=-\frac{B}{A}$ we infer $C>0$, respectively $B<-A<0$.
The inequalities $\omega=\frac{-B+\sqrt{\Delta}}{2A} > 1 > \omega^* =\frac{-B-\sqrt{\Delta}}{2A} >0$ yield
\begin{equation}\label{eq2.9}
-\sqrt{\Delta} < 2A+B  <\sqrt{\Delta} .
\end{equation}
After squaring and dividing by $A$, this leads to $0<A+C< -B$. Squaring again we get $(A-C)^2 <\Delta$, or
\begin{equation}\label{eq2.10}
-\sqrt{\Delta} < A-C < \sqrt{\Delta} .
\end{equation}
Upon \eqref{eq2.9} and \eqref{eq2.10}, we can write $2A+B=\alpha$, $A-C=\beta$ with $\lvert\alpha\rvert,\lvert\beta\rvert < \sqrt{\Delta}$,
and so $(\alpha-2A)^2 =B^2 =\Delta+4A(A-\beta)$, or equivalently $A(\beta-\alpha)=\Delta -\alpha^2$.
Since $\Delta$ is not a square (so $\beta \neq \alpha$), this gives $0<A =\frac{\lvert \Delta-\alpha^2\rvert}{\lvert\beta -\alpha \rvert} \leq \Delta$,
and \eqref{eq2.9} and \eqref{eq2.10} show that $B$ and $C$ can also take only finitely many values.
\end{proof}

\begin{lemma}\label{lem13}
If $\omega \in \RRR_{\e} (\Delta)$, then $\beta:=T_E (\omega) \in \RRR_{\e} (\Delta)$.
\end{lemma}

\begin{proof}
Let $\omega =[(a_1,e_1),(a_2,e_2),\ldots ]\in \RRR_{\e} (\Delta)$. Then
$\beta =\Big( \begin{smallmatrix} 0 & 1 \\ e_1 & -e_1 a_1 \end{smallmatrix}\Big) \omega$ and
$\operatorname{disc}(\beta)=\operatorname{disc} (\omega)=\Delta$. From $\omega=a_1+\frac{e_1}{\beta}$
it follows that $a_1+\frac{e_1}{\beta^*} =\omega^* \in (-1,1)$, showing $\frac{e_1}{\beta^*} <1-a_1$. Hence
$\frac{1}{\lvert \beta^*\rvert} > a_1-1 \geq 1$, showing $\beta^* \in (-1,1)$.
\end{proof}

\begin{proposition}\label{prop14}
For every $\omega \in \X(\Delta)$ the following are equivalent:
\begin{itemize}
\item[(i)] $\omega\in \RRR_{\e} (\Delta)$ .
\item[(ii)] $ECF(\omega)$ is periodic.
\item[(iii)] $\omega$ is a periodic point of the map $T_E$.
\end{itemize}
\end{proposition}

\begin{proof}
(ii) $\Longrightarrow$ (i) follows from Lemma \ref{lem10}. If $ECF(\omega)$ is periodic,
i.e. $\omega =[\,\overline{(a_1,e_1),\ldots,(a_n,e_n)}\,]>1$,
then $\omega^* =-\llangle \,\overline{(a_n,e_n),\ldots,(a_1,e_1)}\,\rrangle \in (-1,1)$,
hence $\omega \in \RRR_{\e} (\Delta)$.

(i) $\Longrightarrow$ (ii) Let $\omega =[(a_1,e_1),(a_2,e_2),\ldots ] \in \RRR_{\e} (\Delta)$.
Since the set $\RRR_{\e} (\Delta)$ is finite, there exist $r\geq 0$ and $n\geq 1$ such that $\Te^r (\omega) =\Te^{r+n} (\omega)$.
The uniqueness of ECF-expansions then shows $(a_s,e_s)=(a_{s+n},e_{s+n})$, $\forall s\geq r+1$.

It remains to show that if $k\geq 1$ and $\Te^k (\omega) =\Te^{k+n} (\omega)$, then
$\Te^{k-1}(\omega) =\Te^{k+n-1}(\omega)$. Set $\omega_i:=\Te^i (\omega)$.
By Lemma \ref{lem13}, $\omega_{k-1}=[(a_k,e_k),\omega_k] =a_k +\frac{e_k}{\omega_k} \in \RRR_E(\Delta)$,
whence $\omega_{k-1}^* =a_k+\frac{e_k}{\omega_k^*} \in (-1,1)$.
Similarly, from $\omega_{k+n-1}=[(a_{k+n},e_{k+n}),\omega_{k+n}]=a_{k+n}+\frac{e_{k+n}}{\omega_{k+n}}$
and $\omega_{k+n} \in \RRR_{\e} (\Delta)$ it follows that
$\omega_{k+n-1}^*=a_{k+n}+\frac{e_{k+n}}{\omega_{k+n}^*} \in (-1,1)$.
But $\omega_k=\omega_{k+n}$ entails $\omega_k^*=\omega_{k+n}^*=:\beta \in (-1,1)$, and thus
$-1 < a_k +\frac{e_k}{\beta} <1$ and $-1 < a_{k+n}+\frac{e_{k+n}}{\beta} < 1$, or equivalently
\begin{equation*}
e_k a_k, e_{k+n} a_{k+n}  \in \bigg( -1-\frac{1}{\beta}, 1 -\frac{1}{\beta}\bigg) .
\end{equation*}
Since $a_k$ and $a_{k+n}$ are even, this gives $(a_k,e_k)=(a_{k+n},e_{k+n})$.
Letting $k$ decrease by one (unless $k=0$), one finds $\omega =\Te^n (\omega)$, showing
that $ECF(\omega)$ is periodic.

(ii) $\Longleftrightarrow$ (iii) follows from the first equality in \eqref{eq2.3}.
\end{proof}

The ECF version of Lagrange's theorem is well known.
It also holds for larger classes of continued fractions, including backward continued fractions.

\begin{proposition}[\cite{KL,Pa,DH}]\label{prop15}
For every $u \in [1,\infty) \setminus \Q$ the following are equivalent:
\begin{itemize}
\item[(i)]
$u$ is a QI.
\item[(ii)]
$ECF(u)$ is eventually periodic, i.e.
\begin{equation}\label{eq2.12}
u =[ (a_1,e_1),\ldots ,(a_r,e_r),\, \overline{(a_{r+1},e_{r+1}),\ldots ,(a_{r+n},e_{r+n})}\,].
\end{equation}
\end{itemize}
\end{proposition}

We consider the Theta groups $\wTheta$ and $\Theta$ defined in the introduction.

\begin{lemma}\label{lem16}
For every $u \in \X(\Delta)$, $\wTheta u \cap \RRR_{\e} (\Delta) \neq \emptyset$.
\end{lemma}

\begin{proof}
The previous proposition allows us to take $u$ as in \eqref{eq2.12}. Setting
\begin{equation*}
\sigma :=\left( \begin{matrix} 0 & 1 \\ e_r & -e_r a_r \end{matrix}\right) \cdots
\left( \begin{matrix} 0 & 1 \\ e_1 & -e_1 a_1 \end{matrix}\right) \in \wTheta ,
\end{equation*}
we have
\begin{equation*}
\sigma u =[\,\overline{(a_{r+1},e_{r+1}),\ldots ,(a_{r+n},e_{r+m})}\,] \in \wTheta u \cap \RRR_{\e} (\Delta) .\qedhere
\end{equation*}
\end{proof}

\subsection{Some bijections between subsets of ${\mathscr R}_E$ and subsets of $\operatorname{SL}(2,\Z)$}\label{sub2.3}

Consider the collections of matrices
\begin{equation*}
\begin{split}
\PP & :=\big\{ M(a_1,e_1)\cdots M(a_n,e_n) \mid n\geq 1, \ a_i \in 2\N ,\  e_i \in \{ \pm 1\} \big\} \quad \mbox{\rm and} \\
\SSS & := \left\{ \sigma = \left( \begin{matrix} p^ \prime & pe \\ q^\prime & qe \end{matrix} \right) \in \operatorname{GL}(2,\Z) \ \bigg| \
\begin{matrix} \sigma \equiv I_2 \ \mbox{\rm or}\ J_2 \pmod{2}, \
e\in \{ \pm 1\} \\  p^\prime >p>q>0,\ p^\prime > q^\prime >q \end{matrix} \right\} .
\end{split}
\end{equation*}

\begin{proposition}\label{prop17}
$\PP=\SSS .$
\end{proposition}

\begin{proof}
($\subseteq$) The matrix $M(a_1,e_1)\cdots M(a_n,e_n)=\Big( \begin{smallmatrix} p_n & p_{n-1}e_n \\ q_n & q_{n-1} e_n \end{smallmatrix} \Big)$
is $\equiv I_2$ or $J_2 \pmod{2}$ as a product of matrices $M(a_i,e_i) \equiv J_2 \pmod{2}$. We also have $q_1=1>q_0=0$, and by induction
$q_n =a_n q_{n-1}+e_{n-1}q_{n-2} \geq 2q_{n-1}-q_{n-2} > q_{n-1}$. Similarly, $p_n >p_{n-1}$ as $p_1=a_1>p_0=1$ and
$p_n-q_n \geq p_{n-1}-q_{n-1} \geq p_1-q_1=a_1-1 \geq p_0-q_0=1$, showing $M(a_1,e_1)\cdots M(a_n,e_n)\in\SSS$.

($\supseteq$) Let $\sigma :=\Big( \begin{smallmatrix} p^\prime & pe \\ q^\prime & qe \end{smallmatrix}\Big)\in\SSS$. Consider
$a:=2\big\lfloor \frac{q^\prime}{2q} +\frac{1}{2}\big\rfloor \geq 2$. The inverse of $M(a,e)$ is $M(a,e)^{-1}=(-e) \Big(
\begin{smallmatrix} 0 & -e \\ -1 & a \end{smallmatrix}\Big)$ and
\begin{equation*}
\sigma_0:=\sigma M(a,e)^{-1} =(-e) \left( \begin{matrix} p^\prime & pe \\ q^\prime & qe \end{matrix}\right)
\left( \begin{matrix} 0 & -e \\ -1 & a \end{matrix}\right) =
\left( \begin{matrix} p & p^\prime -ap \\ q & q^\prime -aq \end{matrix}\right) \in \operatorname{GL}(2,\Z),
\end{equation*}
with $\sigma_0 \equiv I_2$ or $J_2 \pmod{2}$.
Upon $\frac{q^\prime}{2q} -\frac{1}{2} < \frac{a}{2} \leq \frac{q^\prime}{2q} +\frac{1}{2}$ we have $-q\leq q^\prime -aq<q$.
Since $(q^\prime,q)=1$, if $q^\prime -aq =-q$ then $q=1$, which in turn implies $q^\prime \equiv q \mod 2$, contradiction. Hence $q^\prime -aq =fq_0$ with
$f=\operatorname{sgn} (q^\prime -aq)\in \{ \pm 1\}$ and $0<q_0<q$.

Let $p_0:=f(p^\prime -ap)\in \Z$. We have $\sigma_0 =\Big( \begin{smallmatrix} p & p_0 f \\ q & q_0 f \end{smallmatrix}\Big)$.
It remains to check the inequalities $p>p_0>0$ and $p_0>q_0$, and to investigate what happens when $q_0=1$.

Upon $\det (\sigma_0)=f(qp_0-pq_0)\in \{ \pm 1\}$ we have $p_0=\frac{pq_0 \pm 1}{q} =\frac{p}{q} q_0 \pm \frac{1}{q} > q_0 -\frac{1}{q}\geq q_0-1$,
showing $p_0\geq q_0 \geq 1$. Since $(p_0,q_0)=1$, then either $p_0>q_0$ or $q_0=1$.
Since $q-q_0\geq 1$ and $p\geq 2$ as $p>q>q_0$, we have $p(q-q_0)>1$, showing $p> \frac{pq_0+1}{q} \geq p_0$.

Lastly, when $q_0=1$ it follows from $\Big( \begin{smallmatrix} p & p_0 \\ q & q_0 \end{smallmatrix}\Big) \equiv
I_2$ or
$J_2 \pmod{2}$ that $p_0,q\in 2\N$ and
\begin{equation*}
\sigma_0 =\left( \begin{matrix} p & p_0 e \\ q & e \end{matrix}\right) =
\left( \begin{matrix} p_0 & f \\ 1 & 0 \end{matrix}\right)
\left( \begin{matrix} q & e \\ 1 & 0 \end{matrix}\right) =M(p_0,f) M(q,e),
\end{equation*}
where $f\in \{ \pm 1\}$ is given by $f=p-p_0 q=pq_0-p_0 q$.
\end{proof}

\begin{remark}\label{rem18}
The defining relations for $q_n$ and $p_n$ lead to $\frac{q_2}{q_1}=a_2$, $\frac{q_3}{q_2}=a_3+\frac{e_2}{a_2}$,
$\frac{p_1}{p_0}=a_1$, $\frac{p_2}{p_1}=a_2+\frac{e_1}{a_1}$,
$\frac{p_3}{p_2}=a_3+\frac{e_2}{a_2+\frac{e_1}{a_1}}$ and
\begin{equation*}
\begin{split}
\frac{q_n}{q_{n-1}} & = a_n +\frac{e_{n-1}}{\frac{q_{n-1}}{q_{n-2}}} =
a_n +\cfrac{e_{n-1}}{a_{n-1}+ \cfrac{e_{n-2}}{\ddots +\cfrac{e_2}{a_2}}}\, ,\qquad \forall n\geq 2 ,\\
\frac{p_n}{p_{n-1}} & = a_n +\frac{e_{n-1}}{\frac{p_{n-1}}{p_{n-2}}} =
a_n +\cfrac{e_{n-1}}{a_{n-1}+ \cfrac{e_{n-2}}{\ddots +\cfrac{e_1}{a_1}}}\, ,\qquad \forall n\geq 1.
\end{split}
\end{equation*}
The latter and Lemma \ref{lem10} yield
\begin{equation*}
-\frac{1}{[\overline{(a_1,e_1),\ldots ,(a_n,e_n)}]^*} =\frac{e_n p_{n}}{p_{n-1}},\qquad \forall n\geq 1.
\end{equation*}
\end{remark}

Given $N\in \N$ and $\alpha,\beta,\beta_1,\beta_2\geq 1$, we introduce the sets
\begin{equation*}
\begin{split}
W_{\e} & := \big\{ (a,e)\mid a\in 2\N ,e\in \{ \pm 1\}\big\} ,\\
\WW_{\e} & := \{ (w_1,\ldots,w_m) \mid  m\geq 1, w_i \in W_{\e}\}  ,\\
\WW_{\e}^+ & := \{ (w_1,\ldots,w_m) \in \WW_{\e} \mid  (-e_1)\cdots (-e_m)=+1 \},\\
\SSS_+ & : =\{ \sigma\in \SSS \mid \det(\sigma)=+1 \} ,
 \\
\SSS (\alpha,\beta;N) & :=\{ \sigma \in \SSS_+ \mid \operatorname{Tr} (\sigma)\leq N,\ p\geq \alpha q,\ p^\prime \geq \beta p \},\quad
\SSS (N) :=\SSS (1,1;N)  ,\\
\TT (\alpha,\beta_1,\beta_2;N) & :=\{ (\omega,k) \mid \omega \in \RRR_{\e}, \
  \omega \geq \alpha, \ \omega^* \in [ -\tfrac{1}{\beta_2},\tfrac{1}{\beta_1}],\
  \operatorname{Tr}(\widetilde{\Omega}(\omega)^k) \leq N,\  k\in \N  \} ,\\
\TT_k(\alpha,\beta_1,\beta_2;N) & := \{ (\omega,k) \mid \omega\in \RRR_{\e},
\   \omega \geq \alpha, \ \omega^* \in [-\tfrac{1}{\beta_2},\tfrac{1}{\beta_1}],\
  \operatorname{Tr}(\widetilde{\Omega}(\omega)^k) \leq N\},\quad k\in \N ,\\
  \TT (N) & :=\TT (1,1,1;N),\qquad \TT_{k}(N):=\TT_k (1,1,1;N) .
\end{split}
\end{equation*}
For given $N,\alpha,\beta_1,\beta_2$, the sets $\TT_k(\alpha,\beta_1,\beta_2;N)$ are disjoint. Therefore
\begin{equation*}
\lvert \TT (\alpha,\beta_1,\beta_2;N)\rvert =\sum\limits_{k\geq 1} \  \lvert \TT_k(\alpha,\beta_1,\beta_2;N)\rvert,\qquad \forall N\geq 1 .
\end{equation*}

We have already proved that the map
\begin{equation*}
\beta_{\e} : \WW_{\e} \rightarrow \SSS, \qquad
\beta_{\e} \big( (a_1,e_1),\ldots ,(a_m,e_m)\big) :=M(a_1,e_1) \cdots M(a_m,e_m)
\end{equation*}
is well defined and onto. Define also the sets
\begin{equation*}
\begin{split}
\WW_{\e}^+(N) & :=\{ w\in \WW_{\e}^+ \mid \operatorname{Tr}(\beta_E(w))\leq N\} ,  \\
\WW_E^+ (\alpha,\beta_1,\beta_2;N) & := \{ w=(w_1,\ldots,w_m)\in \WW_E^+ (N) \mid\omega :=[\, \overline{w_1,\ldots,w_m}\, ] \geq \alpha,\
\omega^* \in [ -\tfrac{1}{\beta_2},\tfrac{1}{\beta_1} ]\} .
\end{split}
\end{equation*}

\begin{proposition}\label{prop19}
{\em (i)} The map $\beta_{\e}$ is one-to-one and
$\beta_{\e} ( \WW_{\e}^+(N))=\SSS (N)$.

{\em (ii)} The map
\begin{equation*}
j_{\e}: \WW_{\e}^+(N)\rightarrow \TT(N),\quad  j_{\e} \big( (a_1,e_1),\ldots ,(a_m,e_m)\big) := \bigg(
\omega=[\,\overline{(a_1,e_1),\ldots,(a_m,e_m)}\,],\frac{m}{\ep} \bigg)
\end{equation*}
is a one-to-one correspondence.
\end{proposition}

\begin{proof}
To check that $\beta_{\e}$ is one-to-one we employ a descending argument. Assuming
\begin{equation*}
M_{(a_1,e_1)} \cdots M_{(a_n,e_n)} =\left( \begin{matrix} p_n & p_{n-1} e_n \\ q_n & q_{n-1} e_n \end{matrix} \right) =
M_{(b_1,f_1)} \cdots M_{(b_k,f_k)} =\left( \begin{matrix} P_k & P_{k-1} f_k \\ Q_k & Q_{k-1} f_k \end{matrix}\right),
\end{equation*}
the previous remark implies $f_k=e_n$ and $\frac{Q_k}{Q_{k-1}}=\frac{q_n}{q_{n-1}}$.
This gives in turn $b_k=a_n$, leading to
$M_{(a_1,e_1)} \cdots M_{(a_{n-1},e_{n-1})} =M_{(b_1,f_1)} \cdots M_{(b_{k-1},f_{k-1}})$.

The equality $\beta_{\e} ( \WW_{\e}^+ (N))=\SSS (N)$ is clear.

(ii) First we check that $j_{\e} (\WW_{\e}^+(N)) \subseteq \TT (N)$.
Let $w:=((a_1,e_1),\ldots,(a_m,e_m))\in \WW_{\e}^+(N)$ and $\omega :=[\overline{(a_1,e_1),\ldots,(a_m,e_m)}] \in \RRR_{\e}$.
Employing the notation introduced after Lemma \ref{lem8}, let
$n:=\operatorname{per} (\omega)$. Set $k:=\frac{m}{n} \in \N$. If $\delta_n=+1$, then $\ell:=\ep =n$ and $\frac{m}{\ell}=k\in\N$.
If $\delta_n =-1$, then $\ell=2n$ and we have $1=(-e_1)\cdots (-e_m)=\delta_n^k=(-1)^k$, so $k$ is even and $\frac{m}{\ell}=\frac{m}{2n}=\frac{k}{2}\in \N$.
On the other hand $\beta_{\e} (w)=\widetilde{\Omega} (\omega)^{m/\ell}$, showing $\operatorname{Tr} (\widetilde{\Omega}(\omega)^{m/\ell} )\leq N$,
and therefore $j_{\e} (w)\in \TT (N)$.

The map $j_{\e}$ is clearly one-to-one. If $w=((a_1,e_1),\ldots,(a_m,e_m)), w^\prime =((b_1,f_1),\ldots,(b_s,f_s))\in \WW_{\e}^+ (N)$
and $j_{\e} (w)=( \omega, \frac{m}{\ep})=j_{\e} (w^\prime) =(\omega^\prime,\frac{s}{\operatorname{eper}(\omega^\prime)})$, then
$\omega=\omega^\prime$ and $m=s$, so clearly $a_i=b_i$, $i=1,\ldots,m$.

To check surjectivity of $j_{\e}$, let $(\omega,k)\in \TT (N)$, $\omega=[\overline{w_1,\ldots,w_n}]$ with
$w_i=(a_i,e_i)$ and $n=\operatorname{per} (\omega)$, $\ell=\ep$. Take $m:=k\ell$ and
$w=(w_1,\ldots,w_\ell,\ldots ,w_1,\ldots,w_\ell)\in \WW_{\e}^+$ with the block $(w_1,\ldots,w_\ell)$ repeating $k$ times.
Since $\beta_{\e} (w)=\widetilde{\Omega}(\omega)^k$, it follows that $\operatorname{Tr}(\beta_{\e} (w))\leq N$. Therefore, $w \in \WW_{\e}^+(N)$  and $j_{\e}(w)=(\omega,\frac{m}{\ell})$.
\end{proof}

\begin{cor}\label{cor20}
 The restriction of $j_E$ gives a bijection between
the sets $\WW_E^+ (\alpha,\beta_1,\beta_2;N)$ and $\TT (\alpha,\beta_1,\beta_2;N)$.
\end{cor}

\section{Even continued fractions and the Pell equation}\label{Pell}
In this section we discuss some connections between $E$-reduced QIs and the Pell equations $t^2-\Delta u^2 =\pm 1$,
extending some well known results about reduced QIs (\cite{Ha,On}). This part of the paper is not directly related to the proofs of Theorems 1 or 4.

Throughout this section we consider $\omega \in \RRR_{\e} (\Delta)$, $\omega =
[\,\overline{(a_1,e_1),\ldots,(a_n,e_n)}\,] >1$ with $n=\operatorname{per} (\omega)$, and minimal
polynomial $AX^2+BX+C$, where $A>0$, $(A,B,C)=1$, and $\Delta =B^2-4AC \equiv 0,1 \pmod{4}$.
As $\Big( \begin{smallmatrix} p_n & p_{n-1} e_n \\ q_n & q_{n-1} e_n \end{smallmatrix}\Big) \omega =\omega$,
one has
\begin{equation*}
\frac{q_n}{A}=\frac{q_{n-1} e_n -p_n}{B}=\frac{-p_{n-1} e_n}{C}=: v\in \Z .
\end{equation*}

When $n$ is odd, both $q_n$ and $p_{n-1}$ are odd, while $q_{n-1}e_n -p_n$ is always even. It follows that $A,C,v$ are odd and $B$ must be even,
so $\Delta \equiv 0 \pmod{4}$. Hence
$\Delta \equiv 1 \pmod{4} \Longrightarrow \operatorname{per}(\omega)$ even.

Consider the abelian groups
\begin{equation*}
\begin{split}
\FFF_\Delta & :=\{ t+u\sqrt{\Delta} \mid t,u\in\Z ,\  t^2 -\Delta u^2 =\pm 1\} \qquad \mbox{\rm and} \\
\FFF_\Delta^+ & :=\{ t+u\sqrt{\Delta} \mid t>0 ,\  t^2-\Delta u^2 =+1\} .
\end{split}
\end{equation*}
It is well known that $\FFF_\Delta^+$ is an infinite cyclic subgroup of $\FFF_\Delta$.

Recall from Lemma \ref{lem9} that the one-to-one group homomorphism $\Lambda_\omega^{\e} : \wTheta_\omega\rightarrow\R$ maps
a matrix $\Big(\begin{smallmatrix} a & b \\ c & d \end{smallmatrix}\Big)$ to its
eigenvalue $c\omega+d$.

\begin{lemma}\label{lem21}
{\em (i)} If $\Delta \equiv 1 \pmod{4}$, then $\Lambda_\omega^{\e} (\wTheta_\omega)=\FFF_\Delta$.
Furthermore, $\sigma \equiv I_2 \pmod{2}$
for every $\sigma \in \wTheta_\omega$.

{\em (ii)} If $\Delta=4\Delta_0$ and $\Delta_0$ is not a perfect square, then $\Lambda_\omega^{\e} (\wTheta_\omega) \subseteq \FFF_{\Delta_0}$.
If in addition $\Delta_0$ is odd and $B\equiv 0\pmod{4}$, then $\Lambda_\omega^{\e} (\wTheta_\omega)=\FFF_{\Delta_0}$.
\end{lemma}

\begin{proof}
(i) The entries of a matrix $\sigma=\Big( \begin{smallmatrix} a & b \\ c & d \end{smallmatrix}\Big)\in \wTheta_\omega$ must satisfy
$\frac{c}{A}=\frac{d-a}{B}=\frac{-b}{C}=v \in \Z$. We have $D:=(a-d)^2+4bc =
(a+d)^2 -4\det (\sigma)=v^2 \Delta $. Since $a-d\equiv 0 \pmod{2}$ and $\Delta$ is odd, it follows that $D=4u_0^2\Delta$ for some
$u_0\in\Z$. Setting $t_0:=\frac{a+d}{2}\in\Z$, one has $c\omega+d =\frac{a+d+\sqrt{D}}{2}=t_0+u_0\sqrt{\Delta}$,
with ${\mathfrak N} (c\omega +d)=t_0^2-\Delta u_0^2 =\det (\sigma)=\pm 1$. This shows $\Lambda_\omega^{\e} (\wTheta_\omega) \subseteq \FFF_\Delta$.

To show that $\Lambda_\omega^{\e}$ is surjective, note first that $2A\omega +B =\sqrt{\Delta}$. Let
$t_0+u_0\sqrt{\Delta} \in \FFF_\Delta$. Take $\sigma:=\Big( \begin{matrix} t_0 -Bu_0 & -2C u_0 \\ 2A u_0 & t_0+Bu_0 \end{matrix}\Big)
\in \operatorname{GL}(2,\Z)$. We have $(t_0-Bu_0)(t_0+Bu_0)=t_0^2 -\Delta u_0^2 -4AC yu_0^2 \equiv 1\pmod{2}$, hence
$\sigma \equiv I_2 \pmod{2}$, showing $\sigma\in \wTheta$.
The equality $\sigma\omega=\omega$ is checked by direct verification. Finally, we see that
$\Lambda_\omega^{\e} (\sigma)=t_0+(2A\omega +B)u_0=t_0+u_0\sqrt{\Delta}$.

(ii) In this case $D=u^2 \Delta=4u^2 \Delta_0$ and $c\omega+d =\frac{a+d}{2}+u\sqrt{\Delta_0} =t_0+u_0 \sqrt{\Delta}$
with $t_0^2-\Delta u_0^2 ={\mathfrak N} (c\omega+d)=\pm 1$, showing $c\omega+d \in \FFF_{\Delta_0}$.

Assume in addition $\Delta_0$ odd and $B \equiv 0\pmod{4}$. Write $B=2B_0$, $B_0\in\Z$.
Let $t_0+u_0\sqrt{\Delta_0}\in \FFF_{\Delta_0}$. Take $\sigma:=\Big( \begin{smallmatrix} t_0-B_0u_0 & -C u_0 \\
Au_0 & t_0+B_0u_0 \end{smallmatrix}\Big)$ with $\det(\sigma)=t_0^2 -\frac{\Delta}{4} u_0^2 =t_0^2 -\Delta_0u_0^2 =\pm 1$.
Again, the equalities $\sigma\omega=\omega$ and $\Lambda_\omega^{\e} (\sigma)=t_0+u_0\sqrt{\Delta_0}$ follow by direct verification,
so it only remains to check that $\sigma\equiv I_2$ or
$J_2 \pmod{2}$. For this, observe that since
$\Delta_0$ is odd and $t_0^2-\Delta_0 u_0^2 =\pm 1$ one has $t_0 \equiv u_0+1 \pmod{2}$.
This shows that when $u_0$ is even, $t_0$ must be odd and so $\sigma\equiv I_2 \pmod{2}$.
When $u_0$ is odd, $t_0$ must be even, so $\sigma\equiv \left( \begin{smallmatrix} -B_0 & -C \\ A & B_0 \end{smallmatrix}\right) \pmod{2}$.
But $B_0^2 -AC=\Delta_0$ is odd and $B_0$ is even, hence $AC$ is odd and consequently
$\sigma\equiv J_2 \pmod{2}$.
\end{proof}

The following statement (\cite[Proposition 3.2]{BV}, see also \cite{KL}) will be useful. The formulation given here on $[1,\infty)$ is immediately transported from
the one on $[0,1]$ by conjugating by the matrix $\left( \begin{smallmatrix} 0 & 1 \\ 1 & 0 \end{smallmatrix}\right)$.

\begin{lemma}\label{lem22}
Let $u \in [1,\infty)\setminus \Q$ and $\sigma =\Big( \begin{smallmatrix} a & b \\ c & d \end{smallmatrix}\Big)
\in \wTheta$. The following are equivalent:
\begin{itemize}
\item[(i)] There exists $m\geq 1$ such that $\sigma =\Big( \begin{smallmatrix} p_m(u) & p_{m-1}(u) e \\ q_m(u) & q_{m-1}(u) e \end{smallmatrix}\Big)$
for some $e\in \{ \pm 1\}$.
\item[(ii)] $\sigma \in \RRR_{\e}^+ (u) \cup \RRR_{\e}^- (u)$, where
\begin{equation*}
\RRR_{\e}^\varepsilon (u):= \bigg\{ \sigma \in \wTheta \ \bigg\vert \
\begin{matrix} a \geq c > \varepsilon d \geq 0 \\ \varepsilon b \geq \varepsilon d \geq 0 \end{matrix}
\ \ \mbox{\rm and} \ \bigg| \frac{du -b}{-cu +a}\bigg| > 1 \bigg\} ,\quad \varepsilon \in \{ \pm 1\}.
\end{equation*}
We also have $e=\operatorname{sgn} (b)=\operatorname{sgn} (d)$.
\end{itemize}
\end{lemma}

\begin{lemma}\label{lem23}
Suppose $\omega =[\,\overline{(a_1,e_1),\ldots,(a_n,e_n)}\,]>1$, where $n=\operatorname{per} (\omega)$. Then
\begin{itemize}
\item[(i)] If $\sigma=\Big( \begin{smallmatrix} a & b \\ c & d \end{smallmatrix} \Big) \in \wTheta_\omega$ and $c\geq d\geq 0$, then
$e_n=+1$ and $\sigma =\Omega_E (\omega)^k$ for some $k\in \N$.
\item[(ii)] If $\sigma=\Big( \begin{smallmatrix} a & -b \\ c & -d \end{smallmatrix} \Big) \in \wTheta_\omega$ and $c\geq d\geq 0$, then
$e_n=-1$ and $\sigma =\Omega_E (\omega)^k$ for some $k\in \N$.
\end{itemize}
\end{lemma}

\begin{proof}
In both cases we have $\sigma^{-1} \omega=\Big( \begin{smallmatrix} \pm d & \mp b \\ -c & a \end{smallmatrix} \Big)\omega=
\pm \frac{d\omega -b}{-c\omega +a} =\omega >1$, showing $\big| \frac{d\omega -b}{-c\omega+a}\big| =\omega >1$.
Note also that $c-d\equiv 1 \pmod{2}$, so $c>d\geq 1$.

(i) From $\big| \omega -\frac{a}{c}\big| =\big| \frac{a\omega+b}{c\omega+d}-\frac{a}{c}\big| =\frac{1}{c(c\omega+d)} < \frac{1}{c}$ we infer
$\lvert c\omega -a\rvert <1$. This shows $c-a <c\omega -a<1$, so $a\geq c\geq 2$. On the other hand, $b\geq \frac{ad-1}{c} \geq d-\frac{1}{c} >d-1$
gives $b\geq d$. By Lemma \ref{lem22}, $\sigma=\Big( \begin{smallmatrix} p_m (\omega) & p_{m-1}(\omega) \\ q_m(\omega) & q_{m-1}(\omega)
\end{smallmatrix}\Big)$ for some $m\geq 1$. Write $m=kn+r$, $k\geq 0$, $0\leq r<n$. Assume $r>0$. Rewriting the previous equality as
 $\sigma=\Omega_E (\omega)^k M(a_1,e_1)\cdots M(a_r,e_r) \Big( \begin{smallmatrix} 1 & 0 \\ 0 & e_m \end{smallmatrix}\Big)$ and employing
 $\Omega_E (\omega)^{-k}\sigma \omega =\omega$, we get $M(a_r,e_r)^{-1} \cdots M(a_1,e_1)^{-1} \omega =e_m \omega$, so
 $[ (a_{r+1},e_{r+1}),\ldots ,(a_n,e_n),\omega ]=e_m \omega >1$. This gives $e_m=+1$ and
 $\omega =[(a_1,e_1),\ldots,(a_r,e_r),\omega]$, contradiction.
 It follows that $r=0$, so $\sigma =\Omega_E (\omega)^k$. This also gives $e_n=e_{kn}= e_m = +1$.

(ii) Similar to (i), employ $c\omega -d > c-d \geq 1$ and $\omega =\frac{a\omega -b}{c\omega-d}$ to get
 $a\geq c\geq 2$ and $b\geq d$. By Lemma \ref{lem22}, $\sigma =\Big( \begin{smallmatrix} p_m (\omega) & -p_{m-1}(\omega) \\ q_m(\omega) & -q_{m-1}(\omega)
\end{smallmatrix}\Big)$ for some $m=kn+r\geq 1$, $0\leq r< n$. As in (i), when $r>0$ this gives
$\sigma =\Omega_E (\omega)^k M(a_1,e_1)\cdots M(a_r,e_r)\Big( \begin{smallmatrix} 1 & 0 \\ 0 & -e_m \end{smallmatrix}\Big)$, leading in turn to
$e_m=-1$ and $\omega=[(a_1,e_1),\ldots, (a_r,e_r),\omega]$, contradiction. Hence $r=0$, $\sigma =\Omega_E (\omega)^k$ and $e_n=-1$.
\end{proof}

With $\omega$ and $n$ as in Lemma \ref{lem23}, the number
\begin{equation*}
\epsilon:=\Lambda_\omega^{\e} (\Omega_E (\omega))=q_n \omega +q_{n-1} e_n
\end{equation*}
can be expressed as
\begin{equation*}
\epsilon =\frac{p_n+q_{n-1}e_n +\sqrt{(p_n+q_{n-1}e_n)^2 -4\delta_n}}{2} \in \FFF_\Delta ,
\end{equation*}
where $\delta_n:=(-e_1)\cdots (-e_n)$. Since $p_n+q_{n-1}e_n >0$, $\forall n\geq 1$, we have
\begin{equation*}
{\mathfrak N} (\epsilon)=1 \ \Longleftrightarrow \ (-e_1)\cdots (-e_n)=1 \ \Longleftrightarrow\ \epsilon \in \FFF_\Delta^+ .
\end{equation*}

\begin{lemma}\label{lem24}
Suppose that $\Delta \equiv 1 \pmod{4}$ and let $\sigma\in \wTheta_\omega$ \footnote{Recall that in this case $n$ must be even
and $\sigma\equiv I_2 \pmod{2}$.} with $\Lambda_\omega^{\e} (\sigma)>1$.
Then $\Lambda_\omega^{\e} (\sigma)=\epsilon^k$ for some $k\geq 1$. In particular, if $n$ is even and $\ep=\operatorname{per}(\omega)$,
then $\epsilon =\Lambda_\omega^{\e} (\Omega_E (\omega))$ is the fundamental unit of $\FFF_\Delta^+$.
\end{lemma}

\begin{proof}
Let $\sigma=\Big( \begin{smallmatrix} a & b \\ c & d \end{smallmatrix}\Big) \in \wTheta_\omega$ with $\eta:=c\omega +d>1$.
Since ${\mathfrak N} (\eta)=\eta\eta^* =\det (\sigma)=\pm 1$, we have $\lvert \eta^*\rvert =\lvert c\omega^* +d \rvert <1$, leading to
$\eta -\eta^* =c(\omega-\omega^*) >0$, and so $c\geq 1$. On the other hand $\omega^*\in (-1,1)$ yields
$-c +d< \eta^*  =c\omega^*+d <c+d$, showing $-c \leq d\leq c$. Actually we have $-c<d<c$ because
$c-d \equiv 1 \pmod{2}$. Two situations can occur:

(a) $d=0$, leading to $c=1$ (so $\eta=\omega$) and $b=\pm 1$. We get $a+\frac{\pm 1}{\omega}=\omega$, giving
$a>1$ and $\omega =[\overline{(a,\pm 1)}]$ with $\epsilon=q_1 \omega+q_0=\omega=\eta$.

(b) $c> \pm d \geq 1$, leading upon Lemma \ref{lem23} to $e_n=\pm 1$ and $\sigma =\Omega_E (\omega)^k$, $k\geq 1$. Therefore we get
$\eta =\Lambda_\omega^{\e} (\sigma)=\Lambda_\omega^{\e} (\Omega_E (\omega)^k) =\Lambda_\omega^{\e} (\Omega_E (\omega))^k =\epsilon^k $.
\end{proof}

\begin{remark}\label{rem25}
Lemma \ref{lem24} also works when $\Delta \equiv 0\pmod{4}$: if $(-e_1)\cdots (-e_n)=1$, then $\epsilon$ is the generator of the
(infinite cyclic) group $\Lambda_\omega^{\e} (\wTheta_\omega)^+
:=\big\{ \Lambda_\omega^{\e} (\sigma) \mid \sigma\in \wTheta_\omega, \det (\sigma) =1,\operatorname{Tr}(\sigma)>0\}$. When $\Delta=4\Delta_0$, $\Delta_0$ odd and
$B\equiv 0 \pmod{4}$, we also have $\Lambda_\omega^{\e} (\wTheta_\omega)^+ =\FFF_{\Delta_0}^+$.
\end{remark}

\section{Backward continued fractions}\label{sect4}
This section is concerned with backward continued fractions and $B$-reduced QIs.
We consider $\llb a_1,a_2,a_3,\ldots \rrb$ as in equation \eqref{eq1.10}.
The Gauss shift $\Tb$ acts on $[1,\infty)\setminus \Q$ as $\Tb (\llb a_1,a_2,a_3,\ldots\rrb) =\llb a_2,a_3,\ldots\rrb$, or,
in different notation,
\begin{equation*}
\Tb (u)= M(a_1,-1)^{-1} u =\left( \begin{matrix} 0 & -1 \\ 1 & -a_1 \end{matrix} \right)u =
\frac{1}{a_1-u}  =\frac{1}{1-\{ u\}},
\end{equation*}
where $a_1=a_1(u) =1+\lfloor u\rfloor\geq 2$.
The natural extension of the endomorphism $\Tb$ is the automorphism of $[1,\infty)\times [0,1]$ given by
\begin{equation*}
\widetilde{T}_{B} (u,v)=\bigg( \Tb (u) ,\frac{1}{a_1-v}\bigg) =\bigg( \frac{1}{a_1-u},\frac{1}{a_1-v}\bigg) ,
\end{equation*}
acting on the digits as a two-sided shift:
\begin{equation*}
\widetilde{T}_{B} (\llb a_1,a_2\ldots\rrb,\llb a_0,a_{-1},\ldots \rrb^{-1}) =
(\llb a_2,a_3,\ldots \rrb,\llb a_1,a_0,a_{-1},\ldots \rrb^{-1}) ,\qquad a_i\in \N, \ a_i\geq 2 .
\end{equation*}

\begin{center}
\begin{figure}
\includegraphics[scale=0.5,bb = 0 10 430 360]{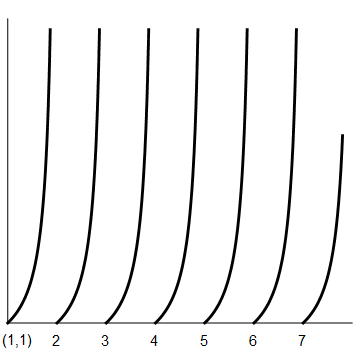}
\includegraphics[scale=0.5,bb = 0 10 280 360]{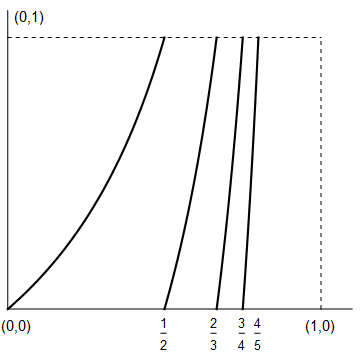}
\caption{The graphs of the maps $\Tb$ and $\oTb$ }\label{Figure2}
\end{figure}
\end{center}

From general ergodic theoretical considerations (\cite{AF}), the infinite measure $(u-v)^{-2} du dv$ is
$\widetilde{T}_{B}$-invariant,
while $d\mu_{B} =\big( \int_0^1 (u-v)^{-2} dv\big) du =\frac{du}{u(u-1)}$ is $\Tb$-invariant.

Conjugating through $J_{B} (x):=\frac{1}{1-x}$, one gets the R\' enyi-Gauss map
$\oTb:=J_{B}^{-1} \Tb J_{\b}$, which acts on $[0,1]$ as
in equation \eqref{eq1.11}
with invariant measure $\nu_{\b} =J_{\b\ast} \mu_{B} = \frac{dx}{x}$ (\cite{Re,AF}).
The endomorphism $\oTb$ is exact in
Rohlin's sense, and $\nu_{\b}$ is the unique $\sigma$-finite Lebesgue absolutely continuous
$\oTb$-invariant measure.

Given $a_i \geq 2$, define $\frac{p_0}{q_0}:=\frac{1}{0}$,
$\frac{p_1}{q_1}:=\frac{a_1}{1}$, $\frac{p_2}{q_2}:= a_1-\frac{1}{a_2}=\frac{a_1a_2-1}{a_2}$, and
\begin{equation*}
\frac{p_k}{q_k}:=\frac{a_k p_{k-1}-p_{k-2}}{a_k q_{k-1}-q_{k-2}} =
a_1 -\cfrac{1}{a_2- \cfrac{1}{\ddots -\cfrac{1}{a_k}}}\, ,\qquad \forall k\geq 2.
\end{equation*}

All algebraic computations with even continued fractions apply to backward continued fractions
taking $e_i=-1$, $\forall i\geq 1$. In particular,
as in Remark \ref{rem18} we have
\begin{equation*}
\begin{split}
\frac{q_n}{q_{n-1}} & = a_n -\frac{1}{\frac{q_{n-1}}{q_{n-2}}} =
a_n -\cfrac{1}{a_{n-1}- \cfrac{1}{\ddots -\cfrac{1}{a_2}}}\, ,\qquad \forall n\geq 2 ,\\
\frac{p_n}{p_{n-1}} & = a_n -\frac{1}{\frac{p_{n-1}}{p_{n-2}}} =
a_n -\cfrac{1}{a_{n-1}- \cfrac{1}{\ddots -\cfrac{1}{a_1}}}\, ,\qquad \forall n\geq 1.
\end{split}
\end{equation*}
We also have
\begin{equation}\label{eq4.1}
p_k>q_k>q_{k-1},\qquad  p_k > p_{k-1}>q_{k-1},
\end{equation}
and quite importantly,
\begin{equation}\label{eq4.2}
p_kq_{k-1}-p_{k-1} q_k =-1,\qquad \forall k\geq 1.
\end{equation}
This shows that the sequence $(\frac{p_k}{q_k})$ is decreasing and $u:=\lim_k \frac{p_k}{q_k} \in [1,\infty)\setminus \Q$ satisfies
\begin{equation}\label{eq4.3}
p_k -q_k u >0,\qquad \forall k\geq 0.
\end{equation}

In the opposite direction, the digits of $u$ are captured by $a_1=a_1(u):=1+\lfloor u\rfloor$ and
\begin{equation*}
a_{n+1}=a_{n+1}(u)=a_1 (\Tb^n (u)) =1+\lfloor \Tb^n (u)\rfloor .
\end{equation*}

The corresponding ECF formulas with $e_i=-1$, $\forall i$ provide
\begin{equation}\label{eq4.4}
\Tb^k (u)=\frac{p_{k-1}-q_{k-1}u}{p_k-q_k u} ,\qquad \forall k\geq 1.
\end{equation}
This gives
\begin{equation*}
u=\frac{\Tb^k (u)p_k -p_{k-1}}{\Tb^k (u) q_k-q_{k-1}} ,\qquad \forall k\geq 1,
\end{equation*}
and so
\begin{equation*}
\Tb (u) \Tb^2 (u)\cdots \Tb^k (u)=\frac{1}{p_k-q_k u},\qquad \forall k\geq 1.
\end{equation*}

\begin{definition}\label{def26}
The quadratic irrational $\omega >1$ is called {\em $B$-reduced} if $0<\omega^* <1$. We denote
\begin{equation*}
\RRR_{\b} (\Delta):=\{ \omega \in \X(\Delta) \mid \omega \ \mbox{$B$-reduced}\}\quad \mbox{and} \quad
\RRR_{\b}:=\bigcup\limits_{\Delta >0} \RRR_{\b} (\Delta) .
\end{equation*}
\end{definition}

Exactly as in Proposition \ref{prop14} one shows

\begin{proposition}\label{prop27}
For every quadratic irrational $\omega >1$, the following are equivalent:
\begin{itemize}
\item[(i)]
$\omega \in \RRR_{\b}$.
\item[(ii)]
$BCF(\omega)$ is periodic, i.e. $\omega =\llb \,\overline{a_1,\ldots,a_n}\,\rrb$.
\item[(iii)]
$\omega$ is a periodic point of the map $T_B$.
\end{itemize}
\end{proposition}

For every $u \in [1,\infty)\setminus \Q$ consider the set
\begin{equation*}
{\mathscr M}_{\b} (u):=\left\{ \sigma =\left( \begin{matrix} p^\prime & -p \\ q^\prime & -q \end{matrix}\right) \in \operatorname{SL}(2,\Z)\ \bigg|\
\begin{matrix} p^\prime >q^\prime >q>0,\  p^\prime >p> q \\
p^\prime -q^\prime u >0,\  E_\sigma (u):=\frac{p-qu}{p^\prime -q^\prime u} > 1 \end{matrix}\right\} .
\end{equation*}

Although the following statement is similar with Lemma \ref{lem22}, we provide a proof for the convenience of the reader.

\begin{lemma}\label{lem28}
Let $u\in [1,\infty)\setminus \Q$.
For every $\sigma =\left( \begin{smallmatrix} p^\prime & -p \\ q^\prime & -q \end{smallmatrix}\right) \in \operatorname{SL}(2,\Z)$, the following are equivalent:
\begin{itemize}
\item[(i)]
$\sigma \in {\mathscr M}_{\b} (u)$.
\item[(ii)]
$\frac{p^\prime}{q^\prime}$ and $\frac{p}{q}$ are consecutive $BCF (u)$-convergents.
\end{itemize}
\end{lemma}

\begin{proof}
(ii) $\Longrightarrow$ (i) Suppose $\frac{p}{q}=\frac{p_n}{q_n}=\frac{p_n(u)}{q_n(u)}$.
Then \eqref{eq4.1}--\eqref{eq4.4} imply $\sigma \in {\mathscr M}_{\b}(u)$.

(i) $\Longrightarrow$ (ii) We prove by induction that the proposition
\begin{equation*}
\begin{split}
P(m):=`` \bigg( \forall u\in [1,\infty)\setminus \Q\bigg) & \bigg( \forall \sigma_0=\left( \begin{matrix}
p & -p_0 \\ q & -q_0 \end{matrix}\right) \in {\mathscr M}_{\b} (u),\ p_0 <m\bigg) \\ &
\bigg( \frac{p}{q}\ \mbox{\rm and}\ \frac{p_0}{q_0}\ \mbox{\rm are consecutive $BCF(u)$-convergents} \bigg)"
\end{split}
\end{equation*}
holds for every $m\geq 2$.

When $m=2$ we get $p_0=1$, $q_0=0$, $q=1$, so $\sigma_0=\Big( \begin{smallmatrix} p & -1 \\ 1 & 0 \end{smallmatrix}\Big)$.
Since $E_{\sigma_0}(u)=\frac{1}{p-u}>1$ and $p-u=p-qu >0$, we infer $p>u>p-1$. This shows that $u=p-\frac{1}{\ddots}$, hence
$\frac{p}{q}=\frac{p}{1}=\frac{p_1}{q_1}$ and $\frac{p_0}{q_0}=\frac{1}{0}$  are consecutive $BCF(u)$-convergents.

Suppose $P(m)$ holds for every $2\leq m\leq m_0$. Let
$\sigma =\Big( \begin{smallmatrix} p^\prime & -p \\ q^\prime & -q \end{smallmatrix} \Big) \in {\mathscr M}_{\b}(u)$ with $p=m_0$.
Then $a:=1+\big\lfloor \frac{p^\prime}{p}\big\rfloor \geq 2$. Let $p_0:=ap-p^\prime$, $q_0:=aq-q^\prime$. From
$a-1 \leq \frac{p^\prime}{p} <a$ we get $p\geq p_0=ap-p^\prime >0$. Actually $p>p_0$ because $p=p_0$ would imply
$p\mid p^\prime$, thus contradicting $(p,p^\prime)=1$.  So $p>p_0 >0$. Since $(p,p^\prime)=1$ we also have $a-1 < \frac{p^\prime}{p}$,
hence $q_0=aq-q^\prime < (\frac{p^\prime}{p}+1) q-q^\prime =\frac{p^\prime q+pq-pq^\prime}{p}=\frac{pq-1}{p} <q$.
On the other hand $q_0> \frac{p^\prime}{p} q-q^\prime=-\frac{1}{p} \geq -\frac{1}{2}$, showing $q>q_0\geq 0$. The inequality $p_0>q_0$ follows from
$p_0-q_0=a(p-q)-(p^\prime-q^\prime) > \frac{p^\prime}{p} (p-q)-(p^\prime -q^\prime) = \frac{pq^\prime-p^\prime q}{p} =\frac{1}{p} >0$,
while $p^\prime -q^\prime u >0$ and $\frac{p^\prime}{q^\prime} < \frac{p}{q}$ yield $p-qu>0$.

Finally, $E_{\sigma_0}(u)=\frac{p_0-q_0 u}{p-qu} >1$ is equivalent to $u> \frac{p-p_0}{q-q_0}$. The latter holds because
$E_\sigma (u)=\frac{p-qu}{p^\prime -q^\prime u} >0$ and $p^\prime -q^\prime u>0$ entail $u>\frac{p^\prime-p}{q^\prime -q}$, while
$\frac{p^\prime-p}{q^\prime -q} \geq \frac{p-p_0}{q-q_0} =\frac{p^\prime -(a-1)p}{q^\prime -(a-1)q}$ is equivalent to
the manifestly true $(a-2)(p^\prime q-pq^\prime)=2-a \leq 0$.
\end{proof}

As a result of \eqref{eq4.2}, here we always have $\widetilde{\Omega}_B (\omega)=\Omega_B (\omega)$.
The $B$-reduced quadratic irrationals $\omega$ will hereby be ordered by the spectral radius
of the matrix $\Omega_B (\omega)$. A precise asymptotic formula for the cardinality of the set
\begin{equation*}
\SSS_{\b} (\alpha,\beta;N):=\left\{ \left( \begin{matrix} p^\prime & -p \\ q^\prime & -q \end{matrix}\right)\in \operatorname{SL}(2,\Z) \
\bigg|\  \begin{matrix}
p\geq \alpha q,\ p^\prime \geq \beta p  \\ p^\prime >q^\prime >q \geq 0,\  p^\prime -q \leq N \end{matrix} \right\} ,
\qquad \alpha,\beta \geq 1,
\end{equation*}
as $N\rightarrow \infty$, will be proved in Section \ref{sect6}.
Finally, the approximation arguments detailed in Section \ref{approx} in the ECF situation will apply
ad litteram to the BCF situation, where all $e_i$'s are equal to $-1$.
This will allow us to conclude that $\lvert \SSS_{\b} (\alpha,\beta;N)\rvert$ provides an accurate approximation of
$r_B(\alpha,\beta;R)$, where $e^R=N^2 \rightarrow \infty$, concluding the proof of Theorem \ref{thm4}.

\section{Some estimates involving Euler's totient sums}
To derive asymptotic estimates for the number of $B$-reduced and $E$-reduced quadratic irrationals $\omega$ with
$(\omega,\omega^\ast) \in [\alpha ,\infty) \times [-\frac{1}{\beta_2},\frac{1}{\beta_1}]$ we will need
some detailed number theoretical estimates involving Euler's totient function.
We consider the following sums:

\begin{equation}\label{eq5.1}
\begin{split}
S_0(N) & :=\sum\limits_{m\leq N} \varphi (m) =\frac{N^2}{2\zeta(2)} +O(N\log N), \\
S_0^O (N) & :=\sum\limits_{\substack{m\leq N \\ m\, \operatorname{odd}}} \varphi (m) =\sum\limits_{a\leq \frac{N}{2}} \varphi (4a),\qquad
S_0^{\e} (N) := \sum\limits_{\substack{m\leq N \\ m\, \operatorname{even}}} \varphi (2m) ,\\
S_1 (N) & := \sum\limits_{m\leq N} \frac{\varphi(m)}{m} = \frac{N}{\zeta (2)}+O(\log N) ,\\
S_1^O (N) & := \sum\limits_{\substack{m\leq N \\ m\, \operatorname{odd}}} \frac{\varphi (m)}{m} =\sum\limits_{a\leq \frac{N}{2}} \frac{\varphi(4a)}{2a}  ,\qquad
S_1^{\e} (N) :=\sum\limits_{\substack{m\leq N \\ m\, \operatorname{even}}} \frac{\varphi (2m)}{m} ,\\
S_2 (N) & :=\sum\limits_{m\leq N} \frac{\varphi (m)}{m^2} =\frac{1}{\zeta (2)} \bigg( \log N+\gamma -\frac{\zeta^\prime (2)}{\zeta(2)} \bigg)
+O (N^{-1}\log N)  , \\
S_2^O (N) & := \sum\limits_{\substack{m\leq N \\ m\, \operatorname{odd}}} \frac{\varphi (m)}{m^2},\qquad
S_2^{\e} (N) :=\sum\limits_{\substack{m\leq N \\ m\, \operatorname{even}}} \frac{\varphi (2m)}{m^2}
=\sum\limits_{a\leq \frac{N}{2}} \frac{\varphi (4a)}{4a^2} .
\end{split}
\end{equation}

The estimates for $S_0(N)$ and $S_1(N)$ are well known. A proof of estimate for $S_2 (N)$ can be found in \cite[Cor. 4.5]{Bo}
(see also Chapter 3 in \cite{Te}). It relies essentially on some version of Perron's integral formula.
Estimates for $S_2^O(N)$ and $S_2^E(N)$ are derived in analogous manner here, but require additional care.

For every positive integer $\ell$, define
\begin{equation*}
C(\ell) := \frac{\varphi(\ell)}{\zeta(2)\ell} \prod\limits_{p\vert \ell} \bigg( 1-\frac{1}{p^2}\bigg)^{-1},
\end{equation*}
with $C(2)=C(4)=\frac{2}{3\zeta(2)}$. By Lemmas 2.1 and 2.2 in \cite{BG} we infer
\begin{equation}\label{eq5.2}
S_0^O (N) =\frac{C(2) N^2}{2}+O (N\log N)
=\frac{N^2}{3\zeta(2)} +O (N\log N),
\end{equation}
\begin{equation}\label{eq5.3}
S_0^{\e} (N) =\sum\limits_{a\leq \frac{N}{2}} \varphi (4a) =
\frac{4C(4)N^2}{8} +O (N\log N) =
\frac{N^2}{3\zeta(2)} +O (N\log N) ,
\end{equation}
\begin{equation}\label{eq5.4}
S_1^O (N) =C(2) N +O (\log^2 N) =
\frac{2N}{3\zeta(2)} +O (\log^2 N) ,
\end{equation}
\begin{equation}\label{eq5.5}
S_1^{\e} (N) =\sum\limits_{a\leq \frac{N}{2}} \frac{\varphi(4a)}{2a}
=\frac{4C(4) N}{4}+O (\log^2 N)
=\frac{2N}{3\zeta(2)} +O (\log^2 N ) .
\end{equation}

To estimate $S_2^O (N)$ and $S_2^{\e} (N)$, we follow closely the proof
of \cite[Lemma 4.4]{Bo} and \cite[Chapter 3]{Te}. We will first estimate the sums
\begin{equation*}
\widetilde{S}_2^O (N):= \sum\limits_{\substack{m\leq N \\ m \,\operatorname{odd}}} \frac{\varphi(m)}{m^2}(N-m)^2 \qquad
\mbox{\rm and} \qquad \widetilde{S}_2^{\e} (N) :=\sum\limits_{\substack{m\leq N \\ m\equiv 0 \pmod{4}}} \frac{\varphi(m)}{m^2}(N-m)^2 ,
\end{equation*}
employing the Perron formula
\begin{equation*}
\frac{1}{\pi i} \int_{\sigma_0 -i\infty}^{\sigma_0+i\infty} \frac{y^s}{s(s+1)(s+2)} \, ds =\begin{cases}
0 & \mbox{\rm if $0\leq y\leq 1$} \\ (1-y^{-1})^2 & \mbox{\rm if $y\geq 1$.} \end{cases} \qquad (\sigma_0 >0)
\end{equation*}

We consider the Dirichlet series
\begin{equation}\label{eq5.6}
\begin{split}
\zeta_O (s) & := \sum\limits_{\substack{m=1 \\ m\, \operatorname{odd}}}^\infty \frac{1}{m^s} =\frac{1}{1+\sum\limits_{k=1}^\infty \frac{\varphi(2^k)}{2^{ks}}}
\prod\limits_p \bigg( 1 + \sum\limits_{k=1}^\infty \frac{\varphi(p^k)}{p^{ks}} \bigg) \\ &
=\frac{2^s-2}{2^s-1} \cdot \frac{\zeta (s-1)}{\zeta(s)}\qquad \mbox{\rm and}
\end{split}
\end{equation}
\begin{equation*}
\begin{split}
\zeta_{\e} (s) & :=\sum\limits_{\substack{m=1 \\ m\equiv 0 \pmod{4}}}^\infty
\frac{1}{m^s} =\frac{\sum\limits_{k=2}^\infty \frac{\varphi(2^k)}{2^{ks}}}{1+\sum\limits_{k=1}^\infty \frac{\varphi(2^k)}{2^{ks}}}
\prod\limits_p \bigg( 1 + \sum\limits_{k=1}^\infty \frac{\varphi(p^k)}{p^{ks}} \bigg)
\\ & =\frac{1}{2^{s-1}(2^s-1)} \cdot \frac{\zeta(s-1)}{\zeta(s)}  \qquad (\operatorname{Re} s >2).
\end{split}
\end{equation*}

Employing the Perron formula with $y=\frac{N}{m}$ and \eqref{eq5.6}, we infer
\begin{equation*}
\begin{split}
\widetilde{S}_2^O (N) & := \sum\limits_{\substack{m\leq N \\ m\, \operatorname{odd}}} \varphi(m)\, \frac{N^2}{m^2} \bigg( 1-\frac{m}{N}\bigg)^2 \\ &
=\frac{1}{\pi i} \int_{\sigma_0-i\infty}^{\sigma_0+i\infty} \sum\limits_{\substack{ m=1 \\ m\, \operatorname{odd}}}^\infty \frac{\varphi(m)}{m^{s+2}} \cdot
\frac{N^{s+2}}{s(s+1)(s+2)}\, ds
 = \frac{1}{2\pi i} \int_{\sigma_0-i\infty}^{\sigma_0+i\infty} g_N^O (s)\, ds, \qquad (\sigma_0>0)
\end{split}
\end{equation*}
where
\begin{equation*}
\begin{split}
g_N^O (s) & =\frac{2N^{s+2}}{s(s+1)(s+2)} \cdot \frac{2^{s+2}-2}{2^{s+2}-1} \cdot \frac{\zeta(s+1)}{\zeta(s+2)} \, ds \\ &
= \frac{2N^{s+2} (2^{s+2}-2)}{(s+1)(s+2)(2^{s+2}-1)\zeta(s+2)} \bigg( \frac{1}{s^2}+\frac{\gamma}{s}+O(1)\bigg) \qquad (s\rightarrow 0)
\end{split}
\end{equation*}
defines a meromorphic function in the region $\operatorname{Re} s > -2$ with a removable singularity at $s=-1$ and a simple pole at $s=0$.

Moving the contour of integration exactly as in \cite[Lemma 4.4]{Bo} we get
\begin{equation*}
\widetilde{S}_2^O (N) =\underset{s=0}{\operatorname{Res}} \ g_N^O (s) =h_N^{O \,  \prime} (0)+O(N),
\end{equation*}
where
\begin{equation*}
h_N^O (s):=\frac{2N^{s+2}(2^{s+2}-2)(1+\gamma s)}{(s+1)(s+2)(2^{s+2}-1)\zeta (s+2)} .
\end{equation*}

Employing the logarithmic derivative of $h_N^O$, we get
\begin{equation}\label{eq5.7}
\widetilde{S}_2^O (N) =\frac{2N^2}{3\zeta(2)} \bigg( \log N +\gamma +\frac{2\log 2}{3}-\frac{3}{2} -\frac{\zeta^\prime(2)}{\zeta(2)}\bigg) +O(N).
\end{equation}

In similar fashion we find
\begin{equation*}
\widetilde{S}_2^{\e} (N)
 = \frac{1}{2\pi i} \int_{\sigma_0-i\infty}^{\sigma_0+i\infty} g_N^E (s)\, ds
 =\underset{s=0}{\operatorname{Res}}\  g_N^{\e} (s) +O(N) =h_N^{\e \,\prime} (0)+O(N),
\end{equation*}
where
\begin{equation*}
\begin{split}
g_N^{\e} (s)  &
= \frac{2N^{s+2}}{(s+1)(s+2)2^{s+1}(2^{s+2}-1)\zeta(s+2)} \bigg( \frac{1}{s^2}+\frac{\gamma}{s}+O(1)\bigg) \quad \mbox{\rm and} \\
h_N^{\e} (s) & =\frac{2N^{s+2} (1+\gamma s)}{(s+1)(s+2)2^{s+1}(2^{s+2}-1)\zeta(s+2)} .
\end{split}
\end{equation*}

Employing the logarithmic derivative of $h_N^{\e}$ we get
\begin{equation}\label{eq5.8}
\widetilde{S}_2^{\e} (N) =\frac{N^2}{6\zeta(2)} \bigg( \log N +\gamma -\frac{7\log 2}{3}-\frac{3}{2} -\frac{\zeta^\prime(2)}{\zeta(2)}\bigg) +O(N).
\end{equation}

Employing $N^2 S_0^O (N)-2N S_1^O (N) +S_0^O (N)=\widetilde{S}_2^O (N)$ and \eqref{eq5.2}, \eqref{eq5.4}, \eqref{eq5.7}, we infer
\begin{equation*}
S_2^O (N)=\sum\limits_{\substack{m\leq N \\ m\, \operatorname{odd}}} \frac{\varphi(m)}{m^2} =
\frac{2}{3\zeta(2)} \bigg( \log N +\gamma +\frac{2\log 2}{3} -\frac{\zeta^\prime (2)}{\zeta(2)}\bigg)
+O (N^{-1}\log^2 N) ,
\end{equation*}
so for every $\theta >1$,
\begin{equation}\label{eq5.9}
S_2^O (N) -S_2^O \bigg( \frac{N}{\theta}\bigg) = \frac{2\log \theta}{3\zeta(2)} +O (N^{-1}\log^2 N) .
\end{equation}

Employing $N^2 S_2^{\e} (N) -2N S_1^{\e} (N)+S_0^{\e} (N)=\widetilde{S}_2^{\e} (2N)$ and \eqref{eq5.3}, \eqref{eq5.5},
\eqref{eq5.8}, we infer
\begin{equation*}
S_2^{\e} (N) =\sum\limits_{\substack{m\leq N \\ m\,\operatorname{even}}} \frac{\varphi(2m)}{m^2}=
\frac{2}{3\zeta(2)} \bigg( \log N +\gamma -\frac{4\log 2}{3} -\frac{\zeta^\prime (2)}{\zeta(2)}\bigg) +O (N^{-1}\log^2 N) ,
\end{equation*}
and so for every $\theta >1$,
\begin{equation}\label{eq5.10}
S_2^{\e} (N) -S_2^{\e} \bigg( \frac{N}{\theta}\bigg) =\frac{2\log \theta}{3\zeta(2)} +O (N^{-1}\log^2 N) .
\end{equation}

\section{Distribution of $B$-reduced quadratic irrationals}\label{sect6}
In this section we prove a precise asymptotic formula for the cardinality of the set $\SSS_B(\alpha,\beta;N)$ considered at the end of Section \ref{sect4}.
First, we will introduce some number theoretical tools that will play a central role in the proof.

For $\Omega \subset \R^2$ and $q,h\geq 1$, denote
\begin{equation*}
\NN_{q,h} (\Omega) := \sum\limits_{\substack{(u,v)\in \Omega \\ uv \equiv h \pmod{q}}} 1,\qquad \NN_q (\Omega) :=\NN_{q,1}(\Omega) .
\end{equation*}
Using Weil bounds for Kloosterman sums one can show (cf., e.g., \cite[Proposition A3]{BZ}) that, if $(h,q)=1$, then for every $I_1,I_2$ intervals,
\begin{equation}\label{eq6.1}
\NN_{q,h}(I_1\times I_2)=\frac{\varphi(q)}{q^2}\, \lvert I_1 \rvert \, \lvert I_2 \rvert
+O_\varepsilon \bigg( q^{1/2+\varepsilon} \Big( 1+\frac{\lvert I_1\rvert}{q}\Big)\Big( 1+\frac{\lvert I_2\rvert}{q}\Big) \bigg),\qquad \forall \varepsilon >0 .
\end{equation}
As long as $(h,q)=1$, the proof of Lemma 2 in \cite{Us} still works when replacing $\pm 1$ by $h$ and employing
$(q,m,nh)=(q,m,n)$, and one gets

\begin{lemma}\label{lem29}
Let $q$, $h$ be integers, $q\geq 2$, $(h,q)=1$. For every integer $c$ and interval $I$ with $\vert I\rvert <q$,
consider the linear function $f(x)=c\pm x$ such that $f(I) \subseteq [0,q]$. Then
\begin{equation*}
\NN_{q,h} (\{ (x,y)\mid x\in I,\ 0\leq y\leq f(x) \})=\frac{\varphi(q)}{q^2} \int_I f(x)\, dx +O_\varepsilon (q^{1/2+\varepsilon}),
\qquad \forall \varepsilon >0 .
\end{equation*}
\end{lemma}

First, we estimate $\NN_{p,-1}(\Omega^-_p (\alpha,\beta;N))$, where
\begin{equation}\label{eq6.2}
\Omega_p^- (\alpha,\beta;N):=\bigg\{ (u,v)\ \bigg| \ u\geq \beta p,\  0\leq v\leq\frac{p}{\alpha} ,\
u-v\leq N \bigg\} .
\end{equation}

When $p>\frac{\alpha N}{\alpha \beta -1}$ we have $\beta p > N+\frac{p}{\alpha}$, and so $\Omega_p^- (\alpha,\beta;N)=\emptyset$.
When $p\leq \frac{\alpha N}{\alpha \beta-1}$ we have $\beta p \leq N+\frac{p}{\alpha}$, giving
\begin{equation}\label{eq6.3}
\operatorname{Area} (\Omega^-_p (\alpha,\beta;N)) = \begin{cases}
(N-\beta p)\frac{p}{\alpha}+\frac{p^2}{2\alpha^2} & \mbox{\rm if $0<p\leq \frac{N}{\beta}$} \\
\frac{1}{2} ( N+\frac{p}{\alpha}-\beta p)^2 & \mbox{\rm if $\frac{N}{\beta} \leq p\leq \frac{\alpha N}{\alpha\beta-1}.$} \end{cases}
\end{equation}

When $p\leq \frac{N}{\beta}$, apply estimate (\ref{eq6.1}) with  $I_1\times I_2 =[\beta p,N] \times [0,\frac{p}{\alpha}]$,
and Lemma \ref{lem29} with $f(x)=x-N$ and $I=[N,N+\frac{p}{\alpha}]$, of length $\frac{p}{\alpha}\leq p$,  together with (\ref{eq6.3}). For the case
$\frac{N}{\beta} \leq p \leq \frac{\alpha N}{\alpha \beta-1}$, apply estimate (\ref{eq6.1}) with
$I_1\times I_2 =[\beta p,N+ \frac{p}{\alpha}] \times [0,\frac{p}{\alpha}]$ and Lemma \ref{lem29} with $f(x)=x-N$ and
$I=[\beta p ,N+\frac{p}{\alpha}]$, of length $\leq \frac{p}{\alpha}$, together with (\ref{eq6.3}), to get

\begin{equation}\label{eq6.4}
\begin{split}
& \lvert \SSS_B  (\alpha,\beta;N)\rvert = \sum\limits_{1\leq p\leq \frac{\alpha N}{\alpha\beta -1}} \NN_{p,-1} (\Omega^-_p (\alpha,\beta;N)) \\ & =
\sum\limits_{1\leq p \leq \frac{ N}{\beta}} \frac{\varphi (p)}{p^2}
\bigg( (N-\beta p)\, \frac{p}{\alpha} +\frac{p^2}{2\alpha^2} \bigg)  +\sum\limits_{\frac{N}{\beta} < p \leq \frac{\alpha N}{\alpha\beta -1}}
\frac{\varphi(p)}{2p^2} \bigg( N+\frac{p}{\alpha}-\beta p\bigg)^2 +O_\varepsilon (N^{3/2+\varepsilon})  \\
& \quad =\frac{N}{\alpha}\, S_1 \bigg( \frac{N}{\beta}\bigg) -\frac{\beta}{\alpha}\, S_0 \bigg( \frac{N}{\beta}\bigg)
+\frac{1}{2\alpha^2} \, S_0 \bigg( \frac{N}{\beta}\bigg)
+\frac{N^2}{2}\bigg( S_2 \bigg( \frac{\alpha N}{\alpha\beta-1}\bigg) -S_2 \bigg( \frac{N}{\beta}\bigg) \bigg) \\
& \qquad +\frac{(1-\alpha\beta)^2}{2\alpha^2} \bigg( S_0 \bigg( \frac{\alpha N}{\alpha\beta -1}\bigg) -S_0 \bigg( \frac{N}{\beta}\bigg)\bigg) \\ & \qquad
+\frac{(1-\alpha\beta)N}{\alpha} \bigg( S_1 \bigg( \frac{\alpha N}{\alpha\beta-1}\bigg) -S_1 \bigg( \frac{N}{\beta}\bigg) \bigg)
+O_\varepsilon (N^{3/2+\varepsilon}) .
\end{split}
\end{equation}

Combining \eqref{eq6.4} with \eqref{eq5.1} we infer
\begin{equation}\label{eq6.5}
\begin{split}
\lvert \SSS_{\b} (\alpha,\beta;N)\rvert & = \frac{N^2}{2\zeta(2)} \, \log \bigg(\frac{\alpha\beta}{\alpha\beta-1}\bigg) +O_\varepsilon (N^{3/2+\varepsilon}) \\
& =\frac{N^2}{2\zeta(2)} \iint_{[\alpha,\infty) \times [0,\frac{1}{\beta}]} \frac{du\, dv}{(u-v)^2} +O_\varepsilon (N^{3/2+\varepsilon}) .
\end{split}
\end{equation}

Theorem \ref{thm4} follows from \eqref{eq6.5} and the approximation arguments\footnote{For the purpose of this
approximation $e_i=-1$ for all $i$. Also, the $a_i$'s are not necessarily even, but only $a_i\geq 2$ is needed.} in Section \ref{approx},
taking $N=e^{R/2}$.
Corollary \ref{cor5} follows taking $\beta=1$.

\section{Estimating the cardinality of the sets $\SSS_\pm (\alpha,\beta;N)$}\label{sect7}

To the end of parameterizing $\SSS(\alpha,\beta;N)$ defined in Subsection \ref{sub2.3}, write
\begin{equation*}
\SSS(\alpha,\beta;N)= \SSS_{+}(\alpha,\beta;N) \cup \SSS_{-}(\alpha,\beta;N),
\end{equation*}
where $\SSS_{i}(\alpha,\beta;N)= \SSS(\alpha,\beta;N) \cap \{e=i \}$, $i= \pm 1$, are disjoint sets.

We start by proving the asymptotic formula
\begin{equation}\label{eq7.1}
\lvert \SSS_{+}(\alpha,\beta;N)\rvert =\frac{N^2}{6\zeta(2)} \log \bigg(\frac{\alpha\beta+1}{\alpha\beta}\bigg) +O_\varepsilon (N^{3/2+\varepsilon}) .
\end{equation}
For every $N\geq 2$, $\alpha,\beta \geq 1$ consider the sets
\begin{equation*}
\begin{split}
\AAA_1 (\alpha,\beta;N) & := \left\{ (m,u,v) \  \bigg| \   \begin{matrix} & u\geq \beta m,\
0\leq v\leq \frac{m}{\beta},\ u+v \leq N  \\ & m\, \operatorname{odd},\  uv\equiv 1 \hspace{-0.8pt} \pmod{m},\ u,v \, \operatorname{even}\end{matrix}  \right\} , \\
\AAA_2 (\alpha,\beta;N) & := \left\{ (m,u,v) \  \bigg| \  \begin{matrix} & u\geq \beta m,\
0\leq v\leq \frac{m}{\alpha},\  u+v \leq N  \\ & m\, \operatorname{even},\  uv\equiv 1 \hspace{-0.8pt} \pmod{2m},\ u,v \, \operatorname{odd}\end{matrix}  \right\}.
\end{split}
\end{equation*}
The last condition in the definition of $\AAA_2(\alpha,\beta;N)$ is obsolete as $uv\equiv 1 \pmod{2m}$ implies $u,v$ odd.

\begin{lemma}\label{lem30}
The map
\begin{equation*}
\Phi:\SSS_{+}(\alpha,\beta;N) \rightarrow \AAA_1(\alpha,\beta;N)\cup \AAA_2(\alpha,\beta;N),\quad\Phi \left( \begin{matrix} p^\prime & p \\
q^\prime & q \end{matrix} \right) := (p, p^\prime ,q)
\end{equation*}
is a bijection.
\end{lemma}

\begin{proof}
The map $\Phi$ is well defined because $\sigma=\left( \begin{smallmatrix} p^\prime & p \\ q^\prime & q \end{smallmatrix}\right)
\equiv I_2$ or $J_2 \pmod{2}$ entails $p$ odd $\Longrightarrow$ $p^\prime,q$ even and $p$ even $\Longrightarrow$ $p^\prime ,q$ odd.
We show that for every $(m,u,v)\in\AAA_1 (\alpha,\beta;N) \cup \AAA_2 (\alpha,\beta;N)$, there exists a unique
$\sigma=\left( \begin{smallmatrix} p^\prime & p \\ q^\prime & q \end{smallmatrix}\right)
\in \SSS_{+}(\alpha,\beta;N)$ such that $\Phi(\sigma)=(m,u,v)$.

Suppose first $(m,u,v)=(p,p^\prime,q) \in \AAA_1(\alpha,\beta;N)$. We must have $p^\prime q-pq^\prime =1$, or, equivalently,
$q^\prime :=\frac{uv-1}{p}$. Since $uv -1$ is odd, we have that $q^\prime $ odd, so $\sigma \equiv J_2 \pmod{2}$.

When $(m,u,v)=(p,p^\prime,q)\in \AAA_2(\alpha,\beta;N)$, we similarly have $q^\prime:=\frac{uv-1}{p}$. The condition $uv \equiv 1\pmod{2p}$ gives $\sigma \equiv I_2 \pmod{2}$.

In both situations we have $v<m<u$, so $v=q\leq \frac{uv-1}{m}=q^\prime < u=p^\prime$.
\end{proof}

As a corollary we obtain
\begin{equation*}
\begin{split}
A_1(\alpha,\beta;N):=\lvert \AAA_1 (\alpha,\beta;N) \rvert & =\sum\limits_{\substack{p\geq 1 \\ p\, \operatorname{odd}}}
\sum\limits_{\substack{u\geq \beta p,\, 0\leq v\leq \frac{p}{\alpha} \\ u+v \leq N \\ uv \equiv 1 \pmod{p} \\
u,v\, \operatorname{even}, \,\frac{uv-1}{p}\, \operatorname{odd}}} 1
=\sum\limits_{\substack{p\geq 1 \\ p\, \operatorname{odd}}}
\sum\limits_{\substack{(u,v)\in\Omega_p^+(\alpha,\beta;N) \\ uv \equiv 1 \pmod{p} \\
u,v\, \operatorname{even}, \,\frac{uv-1}{p}\, \operatorname{odd}}} 1 \quad \mbox{\rm and}
 \\
A_2(\alpha,\beta;N):= \lvert \AAA_2(\alpha,\beta;N)\rvert &
= \sum\limits_{\substack{p\geq 1 \\ p\, \operatorname{even}}}
\sum\limits_{\substack{u\geq \beta p,\ 0\leq v\leq \frac{p}{\alpha} \\ u+v\leq N \\
uv \equiv 1 \pmod{2p}}} 1
= \sum\limits_{\substack{p\geq 1 \\ p\, \operatorname{even}}}
\sum\limits_{\substack{(u,v) \in \Omega_p^+(\alpha,\beta;N) \\
uv \equiv 1 \pmod{2p}}} 1,
\end{split}
\end{equation*}
where we consider the region
\begin{equation*}
\Omega_p^+ (\alpha,\beta;N):= \bigg\{ (u,v)\ \bigg|\
u\geq \beta p,\  0\leq v\leq \frac{p}{\alpha},\ u+v\leq N \bigg\} .
\end{equation*}

When $p>\frac{N}{\beta}$ we have $\Omega^+_p (\alpha,\beta;N)=\emptyset$.
When $p\leq \frac{N}{\beta}$ we have
\begin{equation}\label{eq7.2}
\operatorname{Area} (\Omega^+_p (\alpha,\beta;N)) =\begin{cases}
(N-\beta p)\frac{p}{\alpha} -\frac{p^2}{2\alpha^2} & \mbox{\rm if $0\leq p \leq \frac{\alpha N}{\alpha\beta+1}$} \\
\frac{(N-\beta p)^2}{2} & \mbox{\rm if $\frac{\alpha N}{\alpha\beta+1} \leq p\leq \frac{N}{\beta}$} .
\end{cases}
\end{equation}

Writing $u= 2a, \ v= 2b$, we have that $(u,v) \in \Omega_p^+ (\alpha,\beta;N)\Longleftrightarrow (a,b) \in \tfrac{1}{2} \Omega_p^+ (\alpha,\beta;N)$,
and that $\tfrac{uv-1}{p}= \tfrac{4ab-1}{p} \text{ is odd}
\Longleftrightarrow 4ab-1 \equiv p \pmod{2p} \Longleftrightarrow 4ab \equiv p+1 \pmod{2p} \Longleftrightarrow 2ab \equiv \tfrac{p+1}{2} \pmod{p} \Longleftrightarrow ab \equiv
\overline{2} \cdot  \tfrac{p+1}{2} \pmod{p}$, where $\overline{2} \cdot 2 \equiv 1 \pmod{p}$, so that
  \begin{equation*}
A_1(\alpha,\beta;N) = \sum_{\substack{1\leq p\leq \frac{N}{\beta} \\ p\, \operatorname{odd}}} \mathcal{N}_{p,\overline{2}
 \cdot \tfrac{p+1}{2}} ( \tfrac{1}{2}\Omega_p^+ (\alpha,\beta;N)), \qquad
A_2 (\alpha,\beta;N)  = \sum_{\substack{1\leq p\leq \frac{N}{\beta} \\ p\, \operatorname{even}}} \mathcal{N}_{2p}(\Omega_p^+ (\alpha,\beta;N)).
\end{equation*}

First, we estimate $A_1 (\alpha,\beta;N)$.
Here $p$ is odd,  $(\frac{p+1}{2},p)=1$,
$(\overline{2},p)=1$, so $(\overline{2}\cdot \frac{p+1}{2},p) =1$.
When $p\leq \frac{\alpha N}{\alpha\beta+1}$, we apply estimate \eqref{eq6.1} with
$I_1 \times I_2 =[\beta p,N-\frac{p}{\alpha}]\times [0,\frac{p}{\alpha}]$ and Lemma \ref{lem29} with
$f(x)=N-x$ and $I=[N-\frac{p}{\alpha},N]$, of length $\frac{p}{\alpha}\leq p$, together with \eqref{eq7.2}.
For the case $\frac{\alpha N}{\alpha\beta+1} \leq p\leq \frac{N}{\beta}$, we apply Lemma \ref{lem29} with
$f(x)=N-x$ and $I=[\beta p,N]$, of length $\leq \frac{p}{\alpha}$, together with \eqref{eq7.2}, to get
\begin{equation}\label{eq7.3}
\begin{split}
& A_1(\alpha,\beta;N) =\frac{1}{4} \sum\limits_{\substack{1\leq p\leq \frac{N}{\beta} \\ p\, \operatorname{odd}}}
\frac{\varphi(p)}{p^2}\, \operatorname{Area} (\Omega_p^+ (\alpha,\beta;N)) \\
& =\frac{1}{4} \sum\limits_{\substack{1\leq p\leq \frac{\alpha N}{\alpha\beta+1} \\ p\, \operatorname{odd}}}
\frac{\varphi(p)}{p^2} \bigg( (N-\beta p) \frac{p}{\alpha} -\frac{p^2}{2\alpha^2} \bigg)
+\frac{1}{4} \sum\limits_{\substack{\frac{\alpha N}{\alpha\beta+1} \leq p\leq \frac{N}{\beta} \\ p\, \operatorname{odd}}}
\frac{\varphi(p)}{p^2} \cdot \frac{(N-\beta p)^2}{2} +O_\varepsilon (N^{3/2+\varepsilon}) \\
& \qquad = \frac{N}{4\alpha} \, S_1^O \bigg( \frac{\alpha N}{\alpha\beta+1}\bigg) -\frac{\beta}{4\alpha} \, S_0^O \bigg( \frac{\alpha N}{\alpha\beta+1}\bigg)
-\frac{1}{8\alpha^2}\, S_0^O \bigg( \frac{\alpha N}{\alpha \beta+1}\bigg) \\
& \qquad \qquad + \frac{N^2}{8} \bigg( S_2^O \bigg( \frac{N}{\beta}\bigg) -S_2^O \bigg( \frac{\alpha N}{\alpha \beta+1}\bigg) \bigg)
-\frac{\beta N}{4} \bigg( S_1^O \bigg( \frac{N}{\beta}\bigg) -S_1^O \bigg( \frac{\alpha N}{\alpha\beta+1}\bigg) \bigg) \\
& \qquad \qquad +\frac{\beta^2}{8} \bigg( S_0^O \bigg( \frac{N}{\beta}\bigg) -S_0^O \bigg( \frac{\alpha N}{\alpha\beta+1}\bigg)\bigg)
+O_\varepsilon (N^{3/2+\varepsilon}) .
\end{split}
\end{equation}
Combining \eqref{eq7.3} with \eqref{eq5.2}, \eqref{eq5.3} and \eqref{eq5.9}, we infer after a short calculation
\begin{equation}\label{eq7.4}
A_1 (\alpha,\beta) =\frac{N^2}{12\zeta(2)} \log \bigg(\frac{\alpha \beta +1}{\alpha\beta}\bigg) +O_\varepsilon (N^{3/2+\varepsilon}).
\end{equation}

Next, we estimate $A_2 (\alpha,\beta;N)$.
In this case we have
\begin{equation*}
\begin{split}
& A_2(\alpha,\beta;N) =\frac{1}{4} \sum\limits_{\substack{1\leq p\leq \frac{N}{\beta} \\ p\, \operatorname{even}}}
\frac{\varphi(2p)}{p^2}\, \operatorname{Area} (\Omega_p^+ (\alpha,\beta;N)) \\
& =\frac{1}{4} \sum\limits_{\substack{1\leq p\leq \frac{\alpha N}{\alpha\beta+1} \\ p\, \operatorname{even}}}
\frac{\varphi(2p)}{p^2} \bigg( (N-\beta p) \frac{p}{\alpha} -\frac{p^2}{2\alpha^2} \bigg)
+\frac{1}{4} \sum\limits_{\substack{\frac{\alpha N}{\alpha\beta+1} \leq p\leq \frac{N}{\beta} \\ p\, \operatorname{even}}}
\frac{\varphi(2p)}{p^2} \cdot \frac{(N-\beta p)^2}{2} +O_\varepsilon (N^{3/2+\varepsilon}) \\
& \qquad = \frac{N}{4\alpha} \, S_1^E \bigg( \frac{\alpha N}{\alpha\beta+1}\bigg) -\frac{\beta}{4\alpha} \, S_0^E \bigg( \frac{\alpha N}{\alpha\beta+1}\bigg)
-\frac{1}{8\alpha^2}\, S_0^E \bigg( \frac{\alpha N}{\alpha \beta+1}\bigg) \\
& \qquad \qquad + \frac{N^2}{8} \bigg( S_2^E \bigg( \frac{N}{\beta}\bigg) -S_2^E \bigg( \frac{\alpha N}{\alpha \beta+1}\bigg) \bigg)
-\frac{\beta N}{4} \bigg( S_1^E \bigg( \frac{N}{\beta}\bigg) -S_1^E \bigg( \frac{\alpha N}{\alpha\beta+1}\bigg) \bigg) \\
& \qquad \qquad +\frac{\beta^2}{8} \bigg( S_0^E \bigg( \frac{N}{\beta}\bigg) -S_0^E \bigg( \frac{\alpha N}{\alpha\beta+1}\bigg)\bigg)
+O_\varepsilon (N^{3/2+\varepsilon}) .
\end{split}
\end{equation*}
Combining this with \eqref{eq5.3}, \eqref{eq5.5} and \eqref{eq5.10}, we infer
\begin{equation}\label{eq7.5}
A_2 (\alpha,\beta) =\frac{N^2}{12\zeta(2)} \log \bigg(\frac{\alpha \beta +1}{\alpha\beta}\bigg) +O_\varepsilon (N^{3/2+\varepsilon}).
\end{equation}

The estimate \eqref{eq7.1} follows from \eqref{eq7.4} and \eqref{eq7.5}.

To estimate $\lvert \SSS_{-}(\alpha,\beta;N)\rvert$, consider the region $\Omega_p^- (\alpha,\beta;N)$ defined by \eqref{eq6.2}
and employ $p^\prime q-pq^\prime =-1$, which gives $p^\prime q\equiv -1 \pmod{p}$.
We first observe as above that
\begin{equation*}
\begin{split}
\lvert \SSS_{-}(\alpha,\beta;N)\rvert & =
\sum\limits_{\substack{1\leq p\leq \frac{\alpha N}{\alpha\beta-1} \\ p\, \operatorname{odd}}}
\NN_{p, \overline{2}\cdot \frac{p-1}{2}} (\tfrac{1}{2} \Omega_p^- (\alpha,\beta;N))
+ \sum\limits_{\substack{1\leq p\leq \frac{\alpha N}{\alpha\beta-1} \\ p\, \operatorname{even}}}
\NN_{2p,-1} (\Omega^-_p (\alpha,\beta;N)) \\
& =: B_1 (\alpha,\beta;N) +B_2(\alpha,\beta;N) .
\end{split}
\end{equation*}

Proceeding exactly as in Section \ref{sect6} and as in
the estimation for $A_1(\alpha,\beta;N)$ and $A_2(\alpha,\beta;N)$ above, the only difference being $uv \equiv -1 \pmod{p}$ in place of $uv\equiv 1\pmod{p}$, we show that
\begin{equation*}
\begin{aligned}
B_1 (\alpha,\beta;N) & =\frac{N^2}{12\zeta(2)} \log \bigg(\frac{\alpha\beta}{\alpha\beta-1}\bigg)+O_\varepsilon (N^{3/2+\varepsilon}),\\
B_2(\alpha,\beta;N)&  = \frac{N^2}{12\zeta(2)} \log \bigg(\frac{\alpha\beta}{\alpha\beta-1}\bigg) +O_\varepsilon (N^{3/2+\varepsilon}),
\end{aligned}
\end{equation*}
and therefore
\begin{equation}\label{eq7.6}
\lvert \SSS_{-}(\alpha,\beta;N)\rvert = \frac{N^2}{6\zeta(2)} \log \bigg(\frac{\alpha\beta}{\alpha\beta-1}\bigg)
+O_\varepsilon (N^{3/2+\varepsilon}).
\end{equation}

Combining \eqref{eq7.1} and \eqref{eq7.6}, we get
\begin{equation}\label{eq7.7}
\lvert \SSS_- (\alpha,\beta_1;N)\rvert+\lvert \SSS_+ (\alpha,\beta_2;N)\rvert =
C(\alpha,\beta_1,\beta_2) N^2 +O_\varepsilon (N^{3/2+\varepsilon}),
\end{equation}
with $C(\alpha,\beta_1,\beta_2)$ as in Theorem \ref{thm1}.

\section{Distribution of ECF-reduced quadratic irrationals}\label{approx}
This section completes the proofs of Theorems \ref{thm1} and \ref{thm4} through a careful analysis of the error resulted while
approximating $r_E(\alpha,\beta_1,\beta_2;R)$ by
$\lvert \SSS_- (\alpha,\beta_1;N)\rvert +\lvert \SSS_+(\alpha,\beta_2;N)\rvert$ when $N=e^{R/2}\rightarrow \infty$.

First, we show that the error resulting from replacing the spectral radius of $\widetilde{\Omega}_E (\omega)$ by
the trace is negligible.

\begin{lemma}\label{lem31}
${\mathfrak r} (\widetilde{\Omega}_E (\omega)^k) < \operatorname{Tr} (\widetilde{\Omega}_E (\omega)^k) \leq {\mathfrak r} (\widetilde{\Omega}_E (\omega)^k) +\frac{1}{2}$,
$\quad\forall \omega \in \RRR_{\e}, \forall k\geq 1$.
\end{lemma}

\begin{proof}
First, we show $\eta:={\mathfrak r} (\widetilde{\Omega}_E (\omega)) \geq 2$.
We have ${\mathfrak r} (\Omega_E (\omega)) =q_n \omega +q_{n-1}e_n \geq q_n-q_{n-1}$ and
\begin{equation*}
q_n-q_{n-1} \geq q_{n-1}-q_{n-2} \geq \cdots \geq q_1 -q_0 =1.
\end{equation*}
At least one of the inequalities above is strict, or else $e_1=\cdots =e_n=-1$ and $a_1=\cdots =a_n=2$,
giving $\omega =[\overline{(2,-1)}]=1$, contradiction. We infer ${\mathfrak r} (\Omega_E (\omega)) \geq 2$. The inequality $\eta \geq 2$ follows
replacing $n$ by $2n$ when $\delta_n=-1$.
This leads to $\eta^k < \operatorname{Tr} (\widetilde{\Omega}_E (\omega)^k) =\eta^k +\eta^{-k} \leq \eta^k +\eta^{-1} \leq \eta^k+\frac{1}{2}$.
\end{proof}

\begin{cor}\label{cor32}
$\TT_k(N)=\emptyset$ whenever $k>\log_2 N$.
\end{cor}

\begin{proof}
Let $(\omega,k)\in\TT_k(N)$. We have $\eta^k < \eta^k+\eta^{-k} =\operatorname{Tr} (\widetilde{\Omega}_E (\omega)^k) \leq N$,
giving $k\leq \frac{\log N}{\log \eta} \leq \frac{\log N}{\log 2} =\log_2 N$.
\end{proof}

Denote
\begin{equation*}
T_k (\alpha,\beta_1,\beta_2;N)  :=\lvert \TT_k (\alpha,\beta_1,\beta_2;N)\rvert =
\sum\limits_{\substack{\omega \in \RRR_{\e} \\ \operatorname{Tr} (\widetilde{\Omega}_E (\omega)^k) \leq N \\
\omega \geq \alpha,\,-\frac{1}{\beta_2}\leq \omega^* \leq \frac{1}{\beta_1} }} 1 =T_1 (\alpha,\beta_1,\beta_2;N^{1/k}) .
\end{equation*}
Upon Lemma \ref{lem31} we have, with $r_E(\alpha,\beta_1,\beta_2;R)$ as in \eqref{eq1.9} and $N=e^{R/2}$,
\begin{equation}\label{eq8.1}
T_1 \bigg( \alpha,\beta_1,\beta_2;N-\frac{1}{2}\bigg) \leq r_{\e} (\alpha,\beta_1,\beta_2;R)\leq T_1 (\alpha,\beta_1,\beta_2;N) .
\end{equation}

Assume $\alpha,\beta_1,\beta_1 \geq 1$ with $\alpha \beta_1 >1$. Consider
\begin{equation*}
S (\alpha,\beta_1,\beta_2;N) :=\sum\limits_{k=1}^\infty T_k (\alpha,\beta_1,\beta_2;N) .
\end{equation*}
By Corollary \ref{cor20} we can write\footnote{Recall that $((a_1,e_1),\ldots ,(a_n,e_n))\in \WW_E^+ $ and $\omega:= [ \, \overline{(a_1,e_1),\ldots, (a_n,e_n)}\,]$
imply $\widetilde{\Omega}_E (\omega)=\Omega_E (\omega)$.}
\begin{equation}\label{eq8.2}
S (\alpha,\beta_1,\beta_2;N) =
\vert \WW_E^+(\alpha, \beta_1, \beta_2;N) \vert  =S_-(\alpha,\beta_1;N)+S_+(\alpha,\beta_2;N),
\end{equation}
with $S_-(\alpha,\beta_1;N)$ collecting the contribution of terms with $\omega^* \in (0, \frac{1}{\beta_1}]$ and
$S_+ (\alpha,\beta_2;N)$ collecting the contribution of terms with $\omega^* \in [-\frac{1}{\beta_2},0)$.

By Corollary \ref{cor32} and $T_k (\alpha,\beta_1,\beta_2;N) \ll N^{2/k}$ we infer\footnote{Note that the set $\SSS_+ (1,1;N)$ is finite, but
$\SSS_-(1,1;N)$ and $\SSS_B (1,1;N)$ may be a priori infinite.}
\begin{equation}\label{eq8.3}
\begin{split}
S (\alpha,\beta_1,\beta_2;N) &
=T_1 (\alpha,\beta_1,\beta_2;N)+O \bigg( \sum\limits_{2\leq k\leq \log_2 N} N^{2/k}\bigg) \\
& =T_1 (\alpha,\beta_1,\beta_2;N) +O(N\log N) .
\end{split}
\end{equation}

\begin{lemma}\label{lem33}
{\em (i)} $\displaystyle \Big| \Big\{ \sigma=\Big( \begin{smallmatrix} p^\prime & ep \\ q^\prime & eq \end{smallmatrix}\Big) \in \SSS_+\
\Big\vert \  q(q^\prime +eq)\leq N,\   p^\prime \leq AN\Big\}\Big| = O_{A,\varepsilon}(N^{3/2+\varepsilon})$.\vspace{0.2cm}

{\em (ii)} $\displaystyle \Big| \Big\{ \sigma=\Big( \begin{smallmatrix} p^\prime & -p \\ q^\prime & -q \end{smallmatrix}\Big) \in \SSS_+\  \Big| \
p(p-q)\leq N,\     p^\prime \leq AN \Big\} \Big| = O_{A,\varepsilon}(N^{3/2+\varepsilon})$.\vspace{0.2cm}
\end{lemma}

\begin{proof} (i) When $e=+1$ we get $q\leq \sqrt{N}$ and $pq^\prime =p^\prime q-1 < AN^{3/2}$.
Fix $p^\prime$ and $q$. The number of admissible values for $p$ is at most the number of divisors of $p^\prime q-1$,
so it is $O_{A,\varepsilon} (N^\varepsilon)$. Hence the number of $\sigma$'s is $O_{A,\varepsilon}(N^{3/2+\varepsilon})$.

When $e=-1$ we consider two cases:

(i$_1$) $q\leq \sqrt{N}$. Fixing $p^\prime$ and $q$, the number of admissible values for $p$ is at most the number of divisors of
$p^\prime q+1$, so it is again $O_{A,\varepsilon}(N^\varepsilon )$.

(i$_2$) $q\geq \sqrt{N}$. In this case $0< k:=q^\prime -q \leq \frac{N}{q}\leq \sqrt{N}$ and
$(p^\prime -p)q =pk-1$. Fixing $p$ and $k$, the number of admissible values for $q$ is $O_{A,\varepsilon}(N^\varepsilon)$ as above,
and $p$, $k$ and $q$ completely determine $\sigma$.

(ii) The proof is similar. Consider first $p\leq \sqrt{N}$ and fix $p$ and $q^\prime$, which limits the number of admissible values for
$p^\prime$ to $O_{A,\varepsilon} (N^\varepsilon)$. In the second case $p\geq \sqrt{N}$ gives
$0< \ell:=p-q\leq \frac{N}{p}\leq \sqrt{N}$. We proceed as in (i$_2$). Fix $\ell$ and $q^\prime$ and  observe
that the equality $q(p^\prime -q^\prime)+1=\ell q^\prime$ limits the number of admissible values for $q$ to $O_{A,\varepsilon}(N^\varepsilon)$.
\end{proof}

Employing equality \eqref{eq2.5} and $\Te^n (\omega)=\omega$, we get
\begin{equation}\label{eq8.4}
\bigg| \omega -\frac{p_{n-1}}{q_{n-1}} \bigg| =  \bigg| \frac{e_n p_{n-1}+\omega p_n}{e_n q_{n-1} +\omega q_n} -\frac{p_{n-1}}{q_{n-1}}\bigg| =
\bigg| \frac{e_n}{q_{n-1} (q_n +\frac{e_n}{\omega} q_{n-1})}\bigg| \leq \frac{1}{q_{n-1}(q_n -q_{n-1})}.
\end{equation}
Employing equality \eqref{eq2.7}, we get
\begin{equation}\label{eq8.5}
\bigg| -\frac{1}{\omega^*} -\frac{e_n p_n}{p_{n-1}} \bigg| =
\bigg| \frac{\omega^* e_n q_n -e_n p_n}{\omega^* q_{n-1} -p_{n-1}} -\frac{e_n p_n}{p_{n-1}} \bigg|
=\bigg| \frac{\omega^*}{p_{n-1}(p_{n-1}-\omega^* q_{n-1})}\bigg| \leq \frac{1}{p_{n-1}(p_{n-1}-q_{n-1})}  .
\end{equation}
Furthermore, we have $\omega > \frac{p_{n-1}}{q_{n-1}}$ when $e_n=+1$ and $\omega < \frac{p_{n-1}}{q_{n-1}}$ when $e_n=-1$.
When $e_n=+1$ we also have $-\frac{1}{\omega^*} > \frac{p_n}{p_{n-1}}$, so $\frac{p_n}{p_{n-1}} \geq \beta_2  \Longrightarrow
0> -\omega^* \geq -\frac{1}{\beta_2}$, while when $e_n=-1$ we have $\frac{1}{\omega^*} > \frac{p_n}{p_{n-1}}$, so
$\frac{p_n}{p_{n-1}} \geq \beta_1 \Longrightarrow 0 < \omega^* \leq \frac{1}{\beta_1}$.

Lemma \ref{lem33} will be applied with $A=1$ (making the error term independent of $\beta_2$) for $e=+1$, and with
$A=1+\frac{\alpha}{\alpha\beta_1 -1}$ when $e=-1$.

Combining Lemma \ref{lem33} and \eqref{eq8.2}, \eqref{eq8.4}, \eqref{eq8.5}, \eqref{eq7.7}, we infer
\begin{equation}\label{eq8.6}
\begin{split}
S (\alpha,\beta_1, \beta_2;N) & \leq \sum\limits_{\substack{\sigma \in \SSS_{-} (\alpha,\beta_1-\frac{1}{N};N)\cup \SSS_{+}(\alpha-\frac{1}{N},\beta_2-\frac{1}{N};N) \\
\min\{ q(q^\prime -q),p(p-q)\} \geq N}} 1 +O_{\alpha,\beta_1,\varepsilon} (N^{3/2+\varepsilon}) \\ &
\leq \lvert \SSS_{-} (\alpha,\beta_1-\tfrac{1}{N};N)\rvert
+\lvert \SSS_+ (\alpha-\tfrac{1}{N},\beta_2-\tfrac{1}{N};N)\rvert  +O_{\alpha,\beta_1,\varepsilon} (N^{3/2+\varepsilon}) \\
& =C(\alpha,\beta_1,\beta_2) N^2 + O_{\alpha,\beta_1,\varepsilon}(N^{3/2+\varepsilon}) ,
\end{split}
\end{equation}
and also
\begin{equation}\label{eq8.7}
\begin{split}
S (\alpha,\beta_1, \beta_2;N) & \geq
\sum\limits_{\substack{\sigma \in \SSS_{-} ( \alpha+\frac{1}{N},\beta_1;N)
\cup \SSS_+ ( \alpha,\beta_2;N) \\  \min\{ q(q^\prime -q),p(p-q)\} \geq N }} 1  \\ &
=\lvert \SSS_{-} ( \alpha+\tfrac{1}{N},\beta_1; N ) \rvert
+\lvert \SSS_{+} ( \alpha,\beta_2; N ) \rvert + O_{\alpha,\beta_1,\varepsilon}(N^{3/2+\varepsilon}) \\
& =C(\alpha,\beta_1,\beta_2) N^2 + O_{\alpha,\beta_1,\varepsilon}(N^{3/2+\varepsilon}) .
\end{split}
\end{equation}
Taking $N=e^{R/2}$ and combining \eqref{eq8.1}, \eqref{eq8.3}, \eqref{eq8.6} and \eqref{eq8.7}, we infer
\begin{equation*}
r_E (\alpha,\beta_1,\beta_2;R)=C(\alpha,\beta_1,\beta_2)e^R +O_{\alpha,\beta_1,\varepsilon} (e^{(3/4+\varepsilon)R}) .
\end{equation*}
This proves Theorem \ref{thm1}.

The analysis of the error resulting while
approximating $r_B (\alpha,\beta;R)$ by
$\lvert \SSS_B (\alpha,\beta;N)\rvert $ is similar.
In this case all $e_i$'s are equal to $-1$. This completes the proof of Theorem \ref{thm4}.

\section*{Appendix}
To illustrate the difference between $E$-reduced QIs, $B$-reduced QIs and reduced QIs, we consider a few examples.
We set $a=2k$, $a_1=2k_1$, $a_2=2k_2$, with $k,k_1,k_2 \in \N$.

Here we denote
$$
[b_1,b_2,b_3,\ldots ]:=b_1+\frac{1}{b_2+\cfrac{1}{b_3+\ldots}},\qquad b_i \in \N.
$$
We say that the QI $\omega >1$ is (regular) reduced if $\omega^* \in (-1,0)$.
It is well-known that this is equivalent with $\omega=[\,\overline{b_1,\ldots,b_d}\,]$ for
some $d\geq 1$ and $b_1,\ldots,b_d \in \N$.

\begin{Ex1*}
$\omega =[\, \overline{(a,-1)}\,]=a-\frac{1}{\omega}$ has minimal polynomial $X^2-2kX +1$ and
$\operatorname{disc}(\omega)=4(k^2-1)$.

The $E$-reduced QI
$\omega=k+\sqrt{k^2-1}>1$ is not reduced as $\omega^*=k-\sqrt{k^2-1} \in [0,1]$.

The largest eigenvalue of $\Omega_E(\omega)=\left( \begin{smallmatrix} a & -1 \\ 1 & 0 \end{smallmatrix}\right)$ is
${\mathfrak r}(\Omega_E(\omega))=k+\sqrt{k^2-1}$. Here $\widetilde{\Omega}_E(\omega)=\Omega_E(\omega)^2$ and
so $\varrho_E(\omega)=4\log (k+\sqrt{k^2-1})$.
\end{Ex1*}

\begin{Ex2*}
$\omega =[\, \overline{(a_1,1),(a_2,-1)}\,]=a_1+\frac{1}{a_2-\frac{1}{\omega}}$ has minimal polynomial
$k_2 X^2 -(1+2k_1k_2)X+k_1$ and $\operatorname{disc}(\omega)=4k_1^2k_2^2+1$.

The $E$-reduced QI
$\omega=\frac{2k_1k_2+1+\sqrt{4k_1^2k_2^2+1}}{2k_2}$ is not reduced as
$\omega^*=\frac{2k_1}{2k_1k_2+1+\sqrt{4k_1^2 k_2^2+1}} \in (0,1)$.

The largest eigenvalue of $\Omega_E(\omega)=\left( \begin{smallmatrix} a_1 & 1 \\ 1 & 0 \end{smallmatrix}\right)
\left( \begin{smallmatrix} a_2 & -1 \\ 1 & 0 \end{smallmatrix}\right)=
\left( \begin{smallmatrix} a_1 a_2+1 & -a_1 \\ a_2 & -1 \end{smallmatrix}\right)$ is
${\mathfrak r}(\Omega_E(\omega))=2k_1k_2+\sqrt{4k_1^2k_2^2+1}$. Here $\widetilde{\Omega}_E(\omega)=\Omega_E(\omega)^2$ and
so $\varrho_E(\omega)=4\log (2k_1k_2+\sqrt{4k_1^2k_2^2+1})$.
\end{Ex2*}

\begin{Ex3*}
$\omega =[\, \overline{(a_1,-1),(a_2,1)}\,]=a_1-\frac{1}{a_2+\frac{1}{\omega}}$ has minimal polynomial
$k_2 X^2 +(1-2k_1k_2)X-k_1$ and $\operatorname{disc}(\omega)=4k_1^2k_2^2+1$.

The $E$-reduced QI
$\omega=\frac{2k_1k_2-1+\sqrt{4k_1^2k_2^2+1}}{2k_2}$ is also reduced as
$\omega^*=\frac{-2k_1}{2k_1k_2-1+\sqrt{4k_1^2 k_2^2+1}} \in (-1,0)$.

The largest eigenvalue of $\Omega_E(\omega)=\left( \begin{smallmatrix} a_1 & -1 \\ 1 & 0 \end{smallmatrix}\right)
\left( \begin{smallmatrix} a_2 & 1 \\ 1 & 0 \end{smallmatrix}\right)=
\left( \begin{smallmatrix} a_1 a_2-1 & a_1 \\ a_2 & 1 \end{smallmatrix}\right)$ is
${\mathfrak r}(\Omega_E(\omega))=2k_1k_2+\sqrt{4k_1^2k_2^2+1}$. Here $\widetilde{\Omega}_E(\omega)=\Omega_E(\omega)^2$ and
so $\varrho_E(\omega)=4\log (2k_1k_2+\sqrt{4k_1^2k_2^2+1})$.

Finally notice that the equality
$$
\left( \begin{matrix} a_1 & -1 \\ 1 & 0 \end{matrix}\right) \left( \begin{matrix} a_2 & 1 \\ 1 & 0 \end{matrix}\right)=
\left( \begin{matrix} a_1-1 & 1 \\ 1 & 0 \end{matrix}\right) \left( \begin{matrix} 1 & 1 \\ 1 & 0 \end{matrix}\right)
\left( \begin{matrix} a_2-1 & 1 \\ 1 & 0 \end{matrix}\right)
$$
provides
$$
[\, \overline{(a_1,-1),(a_2,1)}\,] =[\, \overline{a_1-1,1,a_2-1}\,]=a_1-1+\frac{1}{1+\cfrac{1}{a_2-1+\cfrac{1}{a_1-1+\ldots}}}.
$$
In particular $[\, \overline{(2,-1),(2,1)}\,] =[\, \overline{1}\,]=\frac{1+\sqrt{5}}{2}=G$.
We have
$\Omega (G)=\left( \begin{smallmatrix} 1 & 1 \\ 1 & 0 \end{smallmatrix}\right)$, ${\mathfrak r}(\Omega (G))=G$, $\widetilde{\Omega}(G)=\Omega(G)^2$, and
$\varrho (G)=4\log (G)< \varrho_E (G) =4\log (2+\sqrt{5})$.
\end{Ex3*}

\begin{Ex4*}
$\omega =[\, \overline{(a_1,-1),(a_2,-1)}\,]=a_1-\frac{1}{a_2-\frac{1}{\omega}}$ with $a_1\neq a_2$ has minimal polynomial
$\ell_2 X^2 -2d\ell_1 \ell_2 X +\ell_1$ where $d:=(k_1,k_2)$, $k_1=\ell_1 d$, $k_2=\ell_2 d$,
$(\ell_1,\ell_2)=1$, and $\operatorname{disc}(\omega)=4\ell_1\ell_2 (\ell_1 \ell_2 d^2-1)$.

The $E$-reduced QI
$\omega =\ell_1 d+\sqrt{\ell_1^2 d^2 -\frac{\ell_1}{\ell_2}}$ is not
reduced as $\omega^*=\ell_1 d - \sqrt{\ell_1^2 d^2-\frac{\ell_1}{\ell_2}} \in (0,1)$.

The largest eigenvalue of $\Omega_E(\omega)=\left( \begin{smallmatrix} a_1 & -1 \\ 1 & 0 \end{smallmatrix}\right)
\left( \begin{smallmatrix} a_2 & -1 \\ 1 & 0 \end{smallmatrix}\right)=
\left( \begin{smallmatrix} a_1 a_2-1 & -a_1 \\ a_2 & -1 \end{smallmatrix}\right)$ is
${\mathfrak r}(\Omega_E(\omega))=2k_1k_2-1+2\sqrt{k_1k_2(k_1k_2-1)}$. Here $\widetilde{\Omega}_E(\omega)=\Omega_E(\omega)$ and
so $\varrho_E(\omega)=2\log (2k_1k_2-1+2\sqrt{k_1k_2(k_1k_2-1)})$.
\end{Ex4*}

\begin{Ex5*}
$\omega=\llb a,\overline{b}\,\rrb =a-\frac{b}{2}+\frac{\sqrt{b^2-4}}{2}$ with $a\geq 2$, $b>2$, $a\neq b$ has minimal polynomial $X^2 -(2a-b)X+a^2-ab+1$. Then
$\omega^* =a-\frac{b}{2}-\frac{\sqrt{b^2-4}}{2}$ and $\omega^* \notin (0,1)$, or else we get
$b+\sqrt{b^2-4} < 2a<b+\sqrt{b^2-4}$, which yields $a=b$ - contradiction.
\end{Ex5*}

\begin{Ex6*}
The $B$-reduced QI $\omega =\llb \, \overline{3,6}\,\rrb=\frac{3+\sqrt{7}}{2}$ is not reduced as $\omega^* \notin (-1,0)$.

The largest eigenvalue of $\Omega_B(\omega)=\left( \begin{smallmatrix} 3 & -1 \\ 1 & 0 \end{smallmatrix}\right)
\left( \begin{smallmatrix} 6 & -1 \\ 1 & 0 \end{smallmatrix}\right)=
\left( \begin{smallmatrix} 17 & -3 \\ 6 & -1 \end{smallmatrix}\right)$ is
${\mathfrak r}(\Omega_B(\omega)) =8+3\sqrt{7}$ and $\varrho_B(\omega) =2\log (8+3\sqrt{7})$.
Note also that $\sqrt{7}=3-\frac{1}{\omega}=\llb 3, \overline{3,6}\,\rrb$.
\end{Ex6*}

\section*{Acknowledgments}
The research of the second author was partially supported by a 2020 University of Illinois Supplemental Summer Block Grant.

\end{document}